\documentclass[12pt]{article}
\usepackage[a4paper,margin=1.2in]{geometry}
\usepackage{amssymb}
\usepackage{amsmath}
\usepackage{amsfonts}
\usepackage{amsthm}
\usepackage{color}
\usepackage{float}
\usepackage[all]{xy}
\usepackage{scalerel}
\usepackage{stmaryrd}
\usepackage{mathabx}
\usepackage{diagbox}
\usepackage{delimset}

\def\ward{\mathop{\mbox{\textsl{W}}}\nolimits}

\newcommand{\superbinom}{\genfrac{\langle}{\rangle}{0pt}{}}
\newcommand{\numbersp}[2]{\bigg(\!\!\dbinom{#1}{#2}\!\!\bigg)}
\newcommand{\numberspil}[2]{\ensuremath{\big(\!\tbinom{#1}{#2}\!\big)}}

\newcommand{\C}{\mathbb{C}}

\newcommand{\N}{\mathbb{N}}

\newcommand{\R}{\mathbb{R}}

\newtheorem{theorem}{Theorem}

\newtheorem{definition}{Definition}
\newtheorem{proposition}{Proposition}

\newtheorem{example}{Example}

\newtheorem{remark}{Remark}

\begin{document}

\title{\textbf{Ward-Fonten\'e Differential Universal Algebras}}
\author{Ronald Orozco L\'opez}

\newcommand{\Addresses}{{
  \bigskip
  \footnotesize

  \textit{E-mail address}, R.~Orozco: \texttt{rj.orozco@uniandes.edu.co}
  
}}

\maketitle

\begin{abstract}
In this paper a Ward-Fonten\'e differential universal algebra is constructed. In this algebra it is possible to obtain a product $\psi$-rule and a general $\psi$-rule of Leibniz for any calculus on sequences. In particular, the simplicial polytopic calculus and the calculus on Bell numbers are introduced.
\end{abstract}
{\bf Keywords:} algebra of Ward-Fonten\'e, Fonten\'e products, product $\psi$-rule, simplicial polytopic calculus, calculus on Bell number, binomial operator\\
{\bf Mathematics Subject Classification:} 11B65, 11B39, 13F25

\section{Introduction}
Fonten\'e in \cite{font} published a paper in which he generalized the binomial coefficients by replacing $\binom{n}{k}=\frac{n(n-1)\cdots(n-k+1)}{1\cdot2\cdots k}$, consisting of natural numbers, with $\binom{n}{k}_{\psi}=\frac{\psi_{n}\psi_{n-1}\cdots\psi_{n-k+1}}{\psi_{1}\psi_{2}\cdots\psi_{k}}$, formed by an arbitrary sequence $\psi=\{\psi_{n}\}$ of real or complex numbers. He gave a fundamental recurrence relation for these coefficients such that when we make $\psi_{n}=n$ we recover the ordinary binomial coefficients and when we make $\psi_{n}=[n]_{q}=\frac{q^{n}-1}{q-1}$ we recover the $q$-binomial coefficients studied by Gauss, Euler, Jackson and others.

Subsequently, Ward in \cite{ward} developed a calculus on sequences $\psi=\{\psi_{n}\}$ with $\psi_{0}=0$, $\psi_{1}=1$ and $\psi_{n}\neq0$ for all $n\geq1$, and thus generalized the ordinary calculus and the $q$-calculus of Jackson (See \cite{kac}). Other well-studied calculus emerged from his work, the $(p,q)$-calculus and the Fibonomial calculus, where $\psi_{n}=F_{n}$ is the Fibonacci sequence defined recursively by $F_{n+1}=F_{n}+F_{n-1}$, $F_{0}=0$, $F_{1}=1$. For more details on some works on this subject see \cite{benjamin}, \cite{gould}, \cite{kus}, \cite{trojov}, \cite{tuglu}, \cite{pashaev_1}, \cite{pashaev_2}, \cite{pashaev_3}, \cite{ozvatan}.

The $q$-calculus was initiated by Euler \cite{euler} in the 1740s with the study of partitions or additive number theory. Also Gauss, \cite{gauss_1},\cite{gauss_2}, was involved in the $q$-calculus since he studied the $q$-hypergeometric series and the $q$-analogue of binomial formula. However, who actually studied the $q$-calculus systematically was Jackson, \cite{jackson},\cite{jackson_2},\cite{jackson_3},\cite{jackson_4},\cite{jackson_5},\cite{jackson_6},\cite{jackson_7}, introducing the $q$-difference operator, some $q$-functions and the $q$-analogs of the integral. Since then there has been an extensive number of articles and books devoted to the $q$-calculus and its applications in mathematics and physics, for example, in the theory of special functions, difference and differential equations, combinatorics, analytic number theory, quantum theory, quantum group, numerical analysis, operator theory and other related theories. Recently, a new calculus, extension of the $q$-calculus, called the $(p,q)$-calculus or Post Quantum Calculus, was developed by Chakrabarti and Jagannathan \cite{jag}, Brodimas et. al. \cite{brodi}, Wachs and White \cite{wahcs}, and Arik et. al. \cite{arik}. We have the following definition of $(p,q)$-number
\begin{equation}
    [n]_{p,q}=\frac{p^n-q^n}{p-q}.
\end{equation}
When $p=1$, the $(p,q)$-numbers reduce to the $q$-numbers $[n]_{q}$. In general, the $(p,q)$-calculus is reduced to the $q$-calculus when $p=1$. The $(p,q)$-derivative of a function $f$ is defined by
\begin{equation}\label{eqn_pq_diff}
    \mathbf{D}_{p,q}f(x)=
    \begin{cases}
    \frac{f\left(px\right)-f\left(qx\right)}{(p-q)x},&\text{ if } x\neq0;\\
    f^{\prime}(0), &\text{ if } x=0
    \end{cases}
\end{equation}
provided that $f$ is differentiable at 0. When $p=1$, we get the $q$-Jackson derivative
\begin{equation*}
    \mathbf{D}_{q}f(x)=\frac{f(qx)-f(x)}{(q-1)x}.
\end{equation*}

Let $\psi=\{\psi_{n}\}$ be a sequence of complex numbers with $\psi_{0}=0$, $\psi_{1}=1$ and $\psi_{n}\neq0$ for $n\geq2$. We define the $\psi$-factorial numbers as $\psi_{n}!=\psi_{1}\psi_{2}\cdots\psi_{n}$, where $\psi_{0}!=1$. If $\psi_{n}=[n]$, we obtain the $q$-factorials and if $\psi_{n}=F_{n}$, we obtain the $F$-factorials $F_{n}!=F_{1}F_{2}\cdots F_{n}$ with $F_{0}!=1$. The $\psi$-binomial coefficients are given by
\begin{equation*}
    \binom{n}{k}_{\psi}=\frac{\psi_{n}}{\psi_{k}!\psi_{n-k}!}
\end{equation*}
which satisfy the symmetry relation
\begin{equation}\label{eqn_simm}
    \binom{n}{k}_{\psi}=\binom{n}{n-k}_{\psi},\ \ \ 0\leq k\leq n.
\end{equation}
When $\psi_{n}$ is the Fibonacci sequence, the $\psi_{n}$-binomial coefficients are known as Fibonomial coefficients.

Now we define the $\psi$-difference
\begin{equation*}
    \psi_{n}-\psi_{k}=F(n,k)\psi_{n-k}
\end{equation*}
where $F(n,k)$ are appropriate complex numbers. When $\psi_{n}=n$, then $F(n,k)=1$ and when $\psi_{n}=[n]_{q}$, then $F(n,k)=q^k$. The recurrence relation provided by Fonten\'e \cite{font} is given by
\begin{equation}\label{eqn_psi_binom}
    \binom{n+1}{k}_{\psi}=\binom{n}{k-1}_{\psi}+F(n+1,k)\binom{n}{k}_{\psi}.
\end{equation}
If we change $n$ to $n-k$ and apply (\ref{eqn_simm}), we will find that
\begin{equation}\label{eqn_psi_binom_sim}
    \binom{n+1}{k}_{\psi}=\binom{n}{k}_{\psi}+F(n+1,n-k+1)\binom{n}{k-1}_{\psi}
\end{equation}
is the symmetric version of (\ref{eqn_psi_binom}).

On the other hand, if we make $k=n$ in (\ref{eqn_psi_binom}), we will obtain the recurrence relation that the sequence $\psi_{n}$ must fulfill
\begin{equation*}
    \psi_{n+1}=\psi_{n}+F(n+1,n).
\end{equation*}
Thus, we have
\begin{equation*}
    F(n+1,n)=
    \begin{cases}
    1,&\text{ if $\psi_{n}=n$;}\\
    q^{n},&\text{ if $\psi_{n}=[n]_{q}$;}\\
    F_{n-1},&\text{ if $\psi_{n}=F_{n}$.}
    \end{cases}
\end{equation*}

Now we define the $\psi$-derivative as $\mathbf{D}_{\psi}x^n=\psi_{n}x^{n-1}$. When $\psi_{n}=n$, the following product rule is known
\begin{eqnarray*}
    \mathbf{D}(f(x)g(x))&=&(\mathbf{D}f)(x)g(x)+f(x)(\mathbf{D}g)(x)
\end{eqnarray*}
and when $\psi_{n}=[n]_{p,q}$ the product is
\begin{eqnarray*}
    \mathbf{D}_{p,q}(f(x)g(x))&=&(\mathbf{D}_{p,q}f)(x)g(px)+f(qx)(\mathbf{D}_{p,q}g)(x)\\
    &=&(\mathbf{D}_{p,q}f)(x)g(qx)+f(px)(\mathbf{D}_{p,q}g)(x).
\end{eqnarray*}
As will be shown later in this paper, it is possible to introduce two new calculus, one on simplicial polytopic numbers and the other on Bell numbers. However, in neither of them is it possible to construct a product rule based on ordinary products of functions, so the main goal of this paper is to construct new products where such a rule makes sense and where it is possible to obtain the above results. The definition of these products will be based on the numbers $F(n,k)$ and will therefore be called Fonten\'e products. 

This article is divided as follows. In the second section, we introduce two new calculus, the simplicial polytopic calculus and the calculus on Bell numbers. In the third section, the Ward-Fonten\'e 
universal algebras will be constructed. This algebra will have as members the $\psi$-exponential generating functions. On this algebras we will define a family of Fonten\'e products of $\psi$-exponential generating functions by using the numbers $F(n,k)$ as the kernel of each product. Subsequently we show that this family of products define an algebra of products defined on the set of $\psi$-exponential generating functions. In the fourth section, we will use the numbers $F(n,n-k)$ to construct the Ward-Fonten\'e opposites universal algebras. In the following section we will show that the above algebras are differential algebras. In the sixth section, the general $\psi$-rule of Leibniz will be given. Here the binomial coefficients on Fonten\'e products are introduced. Throughout this article we will suppose that $\N$ contains zero.

\section{New calculus on sequences}
In this section we introduce new differential calculus defined on simplicial polytopic numbers and on Bell numbers. The main idea is to show that their respective derivatives satisfy the identity $\mathbf{D}_{\psi}x^n=\psi_{n}x^{n-1}$, but do not satisfy Leibniz's rule using only ordinary products.
\subsection{Simplicial Polytopic Calculus}
The simplicial polytopic numbers are a family of sequences of figurate numbers corresponding to the $d$-dimensional simplex for each dimension $d$, where $d$ is a non-negative integer. For $d$ ranging from $1$ to $5$, we have the following simplicial polytopic numbers, respectively: non-negative numbers $\N$, triangular numbers $\mathrm{T}$, tetrahedral numbers $\mathrm{Te}$, pentachoron numbers $\mathrm{P}$ and hexateron numbers $\mathrm{H}$. A list of the above sets of numbers is as follows:
\begin{align*}
    \N&=\{0,1,2,3,4,5,6,7,8,9,...\},\\
    \mathrm{T}&=\{0,1,3,6,10,15,21,28,36,45,55,66,...\},\\
    \mathrm{Te}&=\{0,1,4,10,20,35,56,84,120,165,...\},\\
    \mathrm{P}&=\{0,1,5,15,35,70,126,210,330,495,715,...\},\\
    \mathrm{H}&=\{0,1,6,21,56,126,252,462,792,1287,...\}.
\end{align*}
The $n$-th simplicial $d$-polytopic numbers are given by the formulae
\begin{equation}
    \binom{n+d-1}{d}=\numbersp{n}{d},
\end{equation}
and
\begin{equation}
    \numbersp{n+1}{d}=\numbersp{n}{d}+\numbersp{n+1}{d-1}.
\end{equation}
From Vandermonde's identity
\begin{equation}\label{eqn_vand}
    \numbersp{n}{d}=\sum_{k=1}^{d}\binom{d-1}{k-1}\binom{n}{k}
\end{equation}
where $\binom{\cdot}{\cdot}$ are the binomial coefficients.

Cigler \cite{cigler} introduced the following definitions.
\begin{definition}
The simplicial $d$-polytorial is defined by
\begin{equation}
    \numberspil{n}{d}!=\prod_{k=1}^{n}\numberspil{k}{d},\ n\geq1,\ \numberspil{0}{d}!=1.
\end{equation}
Let us introduce also the simplicial $d$-polytonomial coefficients
\begin{equation}
    \binom{n}{k}_{S_{d}}=\frac{\numberspil{n}{d}!}{\numberspil{k}{d}!\numberspil{n-k}{d}!}.
\end{equation}
\end{definition}

\begin{definition}
Take $f(x)\in C^{d}(X)$. Define the simplicial $d$-polytopic derivative $\mathbf{D}_{S_{d}}f(x)$ of the function $f(x)$ as
\begin{equation}\label{eqn_der_simp}
    \mathbf{D}_{S_{d}}f(x)=\sum_{k=0}^{d-1}\binom{d-1}{k}\frac{x^{k}}{(k+1)!}\frac{d^{k+1}}{dx^{k+1}}f(x).
\end{equation}
Some specializations are:
\begin{enumerate}
    \item For $d=1$, we obtain the usual derivative $\mathbf{D}_{S_{1}}=\frac{d}{dx}$.
    \item For $d=2$, we define the triangular derivative 
    \begin{equation*}
    \mathbf{D}_{T}=\mathbf{D}_{S_{2}}=\frac{x}{2}\frac{d^2}{dx^2}+\frac{d}{dx}.
    \end{equation*}
    \item For $d=3$, we define the tetrahedral derivative 
    \begin{equation*}
    \mathbf{D}_{Te}=\mathbf{D}_{S_{3}}=\frac{x^2}{6}\frac{d^3}{dx^3}+x\frac{d^2}{dx^2}+\frac{d}{dx}.
    \end{equation*}
    \item For $d=4$, we define the pentachoron derivative
    \begin{equation*}
    \mathbf{D}_{P}=\mathbf{D}_{S_{4}}=\frac{x^3}{24}\frac{d^4}{dx^4}+\frac{x^2}{2}\frac{d^3}{dx^3}+\frac{3x}{2}\frac{d^2}{dx^2}+\frac{d}{dx}.
    \end{equation*}
    \item For $d=5$, we define the hexateron derivative
    \begin{equation*}
    \mathbf{D}_{H}=\mathbf{D}_{S_{5}}=\frac{x^4}{120}\frac{d^5}{dx^5}+\frac{x^3}{6}\frac{d^4}{dx^4}+x^2\frac{d^3}{dx^3}+2x\frac{d^2}{dx^2}+\frac{d}{dx}.
    \end{equation*}
\end{enumerate}
\end{definition}
It is straightforward to show that the derivative $\mathbf{D}_{S_{d}}$ is a linear operator.
\begin{proposition}
For any pair of functions $f$ and $g$ in $C^{d}(X)$ and for all $c\in\R$
\begin{enumerate}
    \item $\mathbf{D}_{S_{r}}(f+g)(x)=(\mathbf{D}_{S_{r}}f)(x)+(\mathbf{D}_{S_{r}}g)(x)$.
    \item $\mathbf{D}_{S_{r}}(cf)(x)=c(\mathbf{D}_{S_{r}}f)(x)$.
\end{enumerate}
\end{proposition}
In the following result, we show that the derivative $\mathbf{D}_{S_{d}}$ satisfies the identity $\mathbf{D}_{\psi}x^n=\psi_{n}x^{n-1}$.
\begin{proposition}
For all $n\geq d$
\begin{equation}
    \mathbf{D}_{S_{d}}x^n=\numbersp{n}{d}x^{n-1}.
\end{equation}
\end{proposition}
\begin{proof}
By using Eqs.(\ref{eqn_vand}) and (\ref{eqn_der_simp}) we reach the expected result. Thus we have
\begin{align*}
    \mathbf{D}_{S_{d}}x^{n}&=\sum_{k=0}^{d-1}\binom{d-1}{k}\frac{x^{k}}{(k+1)!}\frac{d^{k+1}x^n}{dx^{k+1}}\\
    &=\sum_{k=0}^{d-1}\binom{d-1}{k}\frac{x^{k}}{(k+1)!}n(n-1)(n-2)\cdots(n-k)x^{n-k-1}\\
    &=\sum_{k=0}^{d-1}\binom{d-1}{k}\binom{n}{k+1}x^{n-1}\\
    &=\numbersp{n}{d}x^{n-1}.
\end{align*}
\end{proof}
Computing the derivatives $\mathbf{D}_{T}$ and $\mathbf{D}_{Te}$ to the product of functions $fg$ we obtain
\begin{align*}
    \mathbf{D}_{T}(fg)&=(\mathbf{D}_{T}f)g+f(\mathbf{D}_{T}g)+xf^{\prime}g^{\prime},\\
    \mathbf{D}_{Te}(fg)&=(\mathbf{D}_{Te}f)g+f(\mathbf{D}_{Te}g)+\frac{x^2}{2}(f^\prime g^\prime)^\prime+2xf^{\prime}g^{\prime}.
\end{align*}
Then an $S_{d}$-analog for the ordinary derivative of the product of functions cannot be obtained from the above results.

\subsection{Calculus on Bell numbers}
A partition of a set $S$ is defined as family of nonempty, pairwise disjoint subsets of $S$ whose union is $S$. The $n$-th Bell number, denoted by $B_{n}$, is the number of all possible partitions of the $n$-element set $S$. For example, $B_{3}=5$ because the 3-element set $\{1,2,3\}$ can be partitioned in 5 distinct ways:
\begin{align*}
    &\{\{1\},\{2\},\{3\}\},\\
    &\{\{1\},\{2,3\}\},\\
    &\{\{2\},\{1,3\}\},\\
    &\{\{3\},\{1,2\}\},\\
    &\{1,2,3\}.
\end{align*}
The first Bell numbers are:
\begin{equation}
    \{B_{n}\}=\{1,1,2,5,15,52,203,877,...\}.
\end{equation}
Dobinski's formula \cite{dobinski} states that the $n$-th Bell number $B_{n}$ equals 
\begin{equation}
    B_{n}=\frac{1}{e}\sum_{k=0}^{\infty}\frac{k^n}{k!}
\end{equation}
where $e$ denotes the Euler number. We will use Dobinski's formula to define the derivative of the calculus on Bell numbers.

To satisfy Ward's conditions for a calculus on sequences, define $\mathrm{B}_{0}=0$ and $\mathrm{B}_{n}=B_{n-1}$, for $n\geq1$. Next, we will give the Bell-analog of the factorials and binomial coefficients.
\begin{definition}
The Belltorial is defined by
\begin{equation}\label{eqn_fibotorial}
    \mathrm{B}_{n}!=\prod_{k=1}^{n}\mathrm{B}_{k},\ n\geq1,\ \brk[c]{0}_{s,t}!=1.
\end{equation}
Let us introduce also the Bellnomial coefficients
\begin{equation}\label{eqn_fibonomial}
    \binom{n}{k}_{B}=\frac{\mathrm{B}_{n}!}{\mathrm{B}_{k}!\mathrm{B}_{n-k}!}
\end{equation}
\end{definition}
Next we have the Bell-analog of the derivative.
\begin{definition}
For any function $f(x)$, we define the Bell or $B$-derivative of $f(x)$ as
\begin{equation}\label{eqn_der_bell}
    (\mathbf{D}_{B}f)(x)=
    \begin{cases}
    \frac{1}{e}\sum_{k=0}^{\infty}\frac{(\mathbf{D}_{k,0}f)(x)}{k!},&\text{ if }x\neq0;\\
    f^{\prime}(0),&\text{ if }x=0,
    \end{cases}
\end{equation}
where
\begin{equation}
    (\mathbf{D}_{k,0}f)(x)=\frac{f(kx)-f(0)}{kx}
\end{equation}
and $(\mathbf{D}_{0,0}f)(x)=\lim_{k\rightarrow0}(\mathbf{D}_{k,0}f)(x)=f^{\prime}(0)$, provided that $f^{\prime}(0)$ exists.
\end{definition}
Next it is shown that the $B$-derivative is a linear operator.
\begin{proposition}
For any pair of functions $f(x),g(x)$ and for all $c\in\R$ it is satisfied that
\begin{enumerate}
    \item $\mathbf{D}_{B}(cf)(x)=c(\mathbf{D}_{B}f)(x)$.
    \item $\mathbf{D}_{B}(f+g)(x)=(\mathbf{D}_{B}f)(x)+(\mathbf{D}_{B}g)(x)$
\end{enumerate}
\end{proposition}
\begin{proof}
For $x\neq0$
\begin{align*}
    \mathbf{D}_{B}(cf)(x)&=\frac{1}{e}\sum_{k=0}^{\infty}\frac{\mathbf{D}_{k,0}(cf)(x)}{k!}=\frac{c}{e}\sum_{k=0}^{\infty}\frac{(\mathbf{D}_{k,0}f)(x)}{k!}
\end{align*}
and
\begin{align*}
    \mathbf{D}_{B}(f+g)(x)&=\frac{1}{e}\sum_{k=0}^{\infty}\frac{\mathbf{D}_{k,0}(f+g)(x)}{k!}=\frac{1}{e}\sum_{k=0}^{\infty}\frac{(\mathbf{D}_{k,0}f)(x)}{k!}+\frac{1}{e}\sum_{k=0}^{\infty}\frac{(\mathbf{D}_{k,0}g)(x)}{k!}.
\end{align*}
If $x=0$, then the statements follow from the linearity properties of the ordinary derivative.
\end{proof}
Finally we show that $\mathbf{D}_{B}x^n=\mathrm{B}_{n}x^{n-1}$, for all $n\geq0$.
\begin{proposition}
For all $c\in\R$ and for all $n\geq1$
\begin{enumerate}
    \item $\mathbf{D}_{B}c=0$.
    \item $\mathbf{D}_{B}x^n=\mathrm{B}_{n}x^{n-1}$.
\end{enumerate}
\end{proposition}
\begin{proof}
From Eq.(\ref{eqn_der_bell}) follows that
\begin{align*}
    \mathbf{D}_{B}c=\frac{1}{e}\sum_{k=0}^{\infty}\frac{(\mathbf{D}_{k,0}c)(x)}{k!}=0
\end{align*}
and
\begin{align*}
    \mathbf{D}_{B}x^n&=\frac{1}{e}\sum_{k=0}^{\infty}\frac{\mathbf{D}_{k,0}x^n}{k!}\\
    &=\frac{1}{e}\sum_{k=0}^{\infty}\frac{\brk[c]{n}_{k,0}x^{n-1}}{k!}\\
    &=\frac{x^{n-1}}{e}\sum_{k=0}^{\infty}\frac{k^{n-1}}{k!}=B_{n-1}x^{n-1}=\mathrm{B}_{n}x^{n-1}.
\end{align*}
\end{proof}
We use Eq.(\ref{eqn_der_bell}) to obtain the $B$-derivative of the product $fg$ of functions. We have
\begin{align*}
    \mathbf{D}_{B}(fg)(x)&=\frac{1}{e}\sum_{k=0}^{\infty}\frac{\mathbf{D}_{k,0}(fg)(x)}{k!}\\
    &=\frac{1}{e}\sum_{k=0}^{\infty}\frac{f(kx)(\mathbf{D}_{k,0}g)(x)+g(0)(\mathbf{D}_{k,0}f)(x)}{k!}\\
    &=\frac{1}{e}\sum_{k=0}^{\infty}f(kx)\frac{(\mathbf{D}_{k,0}g)(x)}{k!}+g(0)\frac{1}{e}\sum_{k=0}^{\infty}\frac{(\mathbf{D}_{k,0}f)(x)}{k!}\\
    &=\frac{1}{e}\sum_{k=0}^{\infty}f(kx)\frac{(\mathbf{D}_{k,0}g)(x)}{k!}+g(0)(\mathbf{D}_{B}f)(x)
\end{align*}
and
\begin{equation*}
    \mathbf{D}_{B}(fg)(x)=\frac{1}{e}\sum_{k=0}^{\infty}g(kx)\frac{(\mathbf{D}_{k,0}f)(x)}{k!}+f(0)(\mathbf{D}_{B}g)(x).
\end{equation*}
From here it is not possible to obtain the $B$-analogue of the product rule for the ordinary derivative.

\section{Ward-Fonten\'e universal algebra}

Now denote $\ward_{\psi,\C}[[x]]$ the set of formal power series of the form $\sum_{n=0}^{\infty}a_{n}\frac{x^{n}}{\psi_{n}!}$ with coefficients in $\C$. It is clear that $(\ward_{\psi,\C}[[x]],+,\cdot)$ is a ring with sum and product of ordinary series, that is,
$$f(x)+g(x)=\sum_{n=0}^{\infty}(a_{n}+b_{n})\frac{x^{n}}{\psi_{n}!},$$
and
$$f(x)\cdot g(x)=\sum_{n=0}^{\infty}\sum_{k=0}^{n}\binom{n}{k}_{\psi}a_{k}b_{n-k}\frac{x^{n}}{\psi_{n}!},$$ 

where $f(x)=\sum_{n=0}^{\infty}a_{n}\frac{x^{n}}{\psi_{n}!}$, 
$g(x)=\sum_{n=0}^{\infty}b_{n}\frac{x^{n}}{\psi_{n}!}\in\ward_{\psi,\C}[[x]]$. The ring $\ward_{\psi,\C}[[x]]$ will be called $\psi$-ring of Ward of $\psi$-exponential generating functions. For the case $\psi_{n}=n$ the Ward ring reduces to a Hurwitz ring (see \cite{keigher}). When $\psi_{n}=[n]$, then we will call $\ward_{q,\C}[[x]]$ the $q$-ring of Ward. The series $\sum_{n=0}^{\infty}a_{n}\frac{z^{n}}{[n]!}$ are known as Eulerian generating functions. When $\psi_{n}=F_{n}$, we will call $\ward_{F,\C}[[x]]$ the $F$-ring of Ward or Fibonomial ring of Ward. 

When $\psi_{n}=\numberspil{n}{d}$, for $d$ non-negative, define the $S_{d}$-exponential generating functions as
\begin{equation*}
    \sum_{n=0}^{\infty}a_{n}\frac{x^n}{\numberspil{n}{d}!}
\end{equation*}
and we will call to $\ward_{S_{d},\C}[[x]]$ the $S_{d}$-ring of Ward. Likewise, when $\psi_{n}=\mathrm{B}_{n}$ define the $B$-exponential generating functions as
\begin{equation*}
    \sum_{n=0}^{\infty}a_{n}\frac{x^n}{\mathrm{B}_{n}!}
\end{equation*}
and we will call to $\ward_{B,\C}[[x]]$ the $B$-ring of Ward.

Next we define the Fonten\'e product of $\psi$-exponential generating functions.
\begin{definition}
For all $i,j\in\N$ with $j<i$ and for all $f,g\in\ward_{\psi,\C}[[x]]$ we define the product $\ast_{i,j}:\ward_{\psi,\C}[[x]]\times\ward_{\psi,\C}[[x]]\rightarrow\ward_{\psi,\C}[[x]]$ by
\begin{equation}
    f\ast_{i,j}g=\sum_{n=0}^{\infty}\sum_{k=0}^{n}F(n+i,k+j)\binom{n}{k}_{\psi}a_{k}b_{n-k}\frac{x^n}{\psi_{n}!}.
\end{equation}
We also define the product $\ast_{\infty,\infty}:\ward_{\psi,\C}[[x]]\times\ward_{\psi,\C}[[x]]\rightarrow\ward_{\psi,\C}[[x]]$ by 
\begin{equation}
f\ast_{\infty,\infty}g=fg,    
\end{equation}
i.e., $\ast_{\infty,\infty}$ is the ordinary product in $\ward_{\psi,\C}[[x]]$. We will call the products $\ast_{i,j}$ the Fonten\'e products.
\end{definition}
When $\psi_{n}=n$, then the Fonten\'e products reduces to the ordinary product of series in $\ward_{\C}[[x]]$. 
Now let $\ward_{\C}$ denote the ring of sequences with elements in the field $\C$ and let $\rho_{x}$ the map $\rho_{x}:\ward_{\C}\rightarrow\ward_{\psi,\C}[[x]]$ defined by 
\begin{equation*}
\rho_{x}(\{a_{n}\}_{n\in\N})=\sum_{n=0}^{\infty}a_{n}\frac{x^{n}}{\psi_{n}!}.    
\end{equation*}
In other words, $\rho_{x}$ maps the ring of sequences $\ward_{\C}$ to the ring of $\psi$-exponential generating functions $\ward_{\psi,\C}[[x]]$. Let $F_{i,j}$ and $G_{i,j}$ the sequences
\begin{equation*}
    F_{i,j}=\{F(n+i,n+j)\}_{n=0}^{\infty}
\end{equation*}
and
\begin{equation*}
    G_{i,j}=\{F(n+i,j)\}_{n=0}^{\infty}
\end{equation*}
and define the maps $M_{i,j}:\ward_{\psi,\C}[[x]]\rightarrow\ward_{\psi,\C}[[x]]$ by $M_{i,j}(\rho_{x}(\mathbf{a}))=\rho_{x}(F_{i,j}\cdot\mathbf{a})$
and $L_{i,j}:\ward_{\psi,\C}[[x]]\rightarrow\ward_{\psi,\C}[[x]]$ by $L_{i,j}(\rho_{x}(\mathbf{a}))=\rho_{x}(G_{i,j}\cdot\mathbf{a})$ for all $\mathbf{a}\in\ward_{\C}$, where $\cdot$ is the component-wise product of sequences in $\ward_{\C}$. When $j=0$, $G_{i,0}=\{1,1,1,...\}$. That is, the map $L_{i,0}$ is the identity map. With this in mind, we show that the set $\ward_{\psi,\C}[[x]]$ with the Fonten\'e product is a ring.

\begin{theorem}\label{theo_anillo_i_j}
If $\psi_{n}=n$, then $(\ward_{\psi,\C}[[x]],+,\ast_{i,j})$ is a commutative ring with unit. If $\psi_{n}\neq n$, then the set $(\ward_{\psi,\C}[[x]],+,\ast_{i,j})$ is a non-associative and non-commutative ring. For all $f(x)$ in $\ward_{\psi,\C}[[x]]$ it holds that $f(x)\ast_{i,j}e=eM_{i,j}(f(x))$ and $e\ast_{i,j}f(x)=eL_{i,j}(f(x))$ with $e=\frac{1}{F(i,j)}$. When $\psi_{n}=[n]$, $e$ is an unit on left-hand side in $\ward_{q,\C}[[x]]$. In general, if $j=0$, then $1$ is an unit on the left-hand side in $\ward_{\psi,\C}[[x]]$. We will call $(\ward_{\psi,\C}[[x]],+,\ast_{i,j})$ the $\psi$-ring of Ward-Fonten\'e with product $\ast_{i,j}$.
\end{theorem}
\begin{proof}
Clearly $(\ward_{\psi,\C}[[x]],+)$ is an abelian group. If $\psi_{n}=n$, then $(\ward_{\psi,\C}[[x]],+,\ast_{i,j})$ is a commutative ring with unit. Now suppose that $\psi_{n}=[n]$. Then $f(x)\ast_{i,j}g(x)=q^{j}f(qx)g(x)$. On the one hand
\begin{eqnarray*}
(f(x)\ast_{i,j}g(x))\ast_{i,j}h(x)&=&(q^{j}f(qx)g(x))\ast_{i,j}h(x)\\
&=&q^{2j}f(q^{2}x)g(qx)h(x)
\end{eqnarray*}
and on the other hand
\begin{eqnarray*}
f(x)\ast_{i,j}(g(x)\ast_{i,j}h(x))&=&f(x)\ast_{i,j}(q^{j}g(qx)h(x))\\
&=&q^{2j}f(qx)g(qx)h(x).
\end{eqnarray*}
Therefore the product $\ast_{i,j}$ is not associative. It is also easy to show that $\ward_{q,\C}[[x]]$ is not commutative. Now, for all $f(x)\in\ward_{q,\C}[[x]]$ it holds that
\begin{equation}\label{eqn_f_e}
    f(x)\ast_{i,j}e=e\sum_{n=0}^{\infty}F(n+i,n+j)a_{n}\frac{x^n}{\psi_{n}!}=eM_{i,j}(f(x))=f(qx)
\end{equation}
and
\begin{equation}\label{eqn_e_f}
    e\ast_{i,j}f(x)=e\sum_{n=0}^{\infty}F(n+i,j)a_{n}\frac{x^n}{\psi_{n}!}=eL_{i,j}(f(x))=f(x)
\end{equation}
Then $e$ is an unit on the left. Moreover, it is satisfied that $f\ast_{i,j}(g+h)=f\ast_{i,j}g+f\ast_{i,j}h$ and $(f+g)\ast_{i,j}h=f\ast_{i,j}h+g\ast_{i,j}h$. Thus $(\ward_{q,\C}[[x]],+,\ast_{i,j})$ is a ring non-commutative, non-associative with unit on left. We will now do the proof for any sequence other than $[n]$. We will show that the product $\ast_{i,j}$ is non-associative. On the one hand,
\begin{multline*}
    (f\ast_{i,j}g)\ast_{i,j}h=\left(\sum_{n=0}^{\infty}\sum_{k=0}^{n}F(n+i,k+j)\binom{n}{k}_{\psi}a_{k}b_{n-k}\frac{x^n}{\psi_{n}!}\right)\ast_{i,j}\sum_{n=0}^{\infty}c_{n}\frac{x^n}{\psi_{n}!}\\
    =\sum_{n=0}^{\infty}\sum_{l=0}^{n}\sum_{k=0}^{l}F(n+i,l+j)F(n+i,k+j)\binom{n}{l}_{\psi}\binom{l}{k}_{\psi}a_{k}b_{l-k}c_{n-l}\frac{x^n}{\psi_{n}!}
\end{multline*}
and on the other hand
\begin{multline*}
    f\ast_{i,j}(g\ast_{i,j}h)=\sum_{n=0}^{\infty}a_{n}\frac{x^n}{\psi_{n}!}\ast_{i,j}\left(\sum_{n=0}^{\infty}\sum_{k=0}^{n}F(n+i,k+j)\binom{n}{k}_{\psi}b_{k}c_{n-k}\frac{x^n}{\psi_{n}!}\right)\\
    =\sum_{n=0}^{\infty}\sum_{l=0}^{n}\sum_{k=0}^{n-l}F(n+i,l+j)F(n-l+i,k+j)\binom{n}{k}_{\psi}\binom{n-l}{k}_{\psi}a_{l}b_{k}c_{n-k-l}\frac{x^n}{\psi_{n}!}.
\end{multline*}
Now it suffices to note that for $n\geq1$ the coefficients of $(f\ast_{i,j}g)\ast_{i,j}h$ and $f\ast_{i,j}(g\ast_{i,j}g)$ do not match. 
In the same way it is proved that $\ward_{\psi,\C}[[x]]$ is not commutative. Now, making $j=0$ in Eq.(\ref{eqn_f_e}) and Eq.(\ref{eqn_e_f}) we obtain
\begin{equation*}
    1\ast_{i,0}f(x)=\sum_{n=0}^{\infty}F(n+i,0)a_{n}\frac{x^n}{\psi_{n}!}=f(x)
\end{equation*} 
and
\begin{equation*}
    f(x)\ast_{i,0}1=\sum_{n=0}^{\infty}F(n+i,n)a_{n}\frac{x^n}{\psi_{n}!}=M_{i,0}(f(x))
\end{equation*}
and thus $1$ is a unit on the left. If $j\neq0$, then there is no unit with respect to the product $\ast_{i,j}$. Finally, a proof for the distributive property is similar to the one done for the $\psi=[n]$ case. 
\end{proof}

\begin{proposition}
For all $f,g$ in $\ward_{\psi,\C}[[x]]$ and for all $\alpha\in\C$ we have
\begin{enumerate}
    \item $(\alpha f)\ast_{i,j}g=f\ast_{i,j}(\alpha g)=\alpha(f\ast_{i,j}g)$.
    \item If $\psi_{n}=[n]$, then $(\alpha\ast_{i,j}f)\ast_{i,j}g=\alpha\ast_{i,j}(f\ast_{i,j}g)$.
\end{enumerate}
\end{proposition}
\begin{proof}
The proof of 1 is trivial and this result tells us that the $\psi$-ring $\ward_{\psi,\C}[[x]]$ is a $\C$-algebra. To prove 2 we will keep in mind that $(f\ast_{i,j}g)\ast_{i,j}h=q^{2j}f(q^{2}x)g(qx)h(x)$ and $f\ast_{i,j}(g\ast_{i,j}h)=q^{2j}f(qx)g(qx)h(x)$. Then by making $f(x)=\alpha$ we obtain the desired result.
\end{proof}

\begin{remark}
Result 2 in the above proposition does not hold in general for all $\psi$-ring $\ward_{F,\C}[[x]]$. The latter can be checked for the $F$-ring $\ward_{F,\C}[[x]]$. On the other hand, result 1 allows us to obtain a field $\C_{\psi}$ isomorphic to $\C$, where $\alpha\ast_{i,j}\beta=F(i,j)\alpha\beta$ with unit $e=\frac{1}{F(i,j)}$. \end{remark}

If we concatenate Fonten\'e products, we will obtain new products as shown below. 
\begin{definition}
For all $i_{1},i_{2},j_{1},j_{2}$ with $j_{1}<i_{1}$, $j_{2}<i_{2}$ and for all $f,g,\in\ward_{\psi,\C}[[x]]$ we define the product $\ast_{i_{1},j_{1}}\ast_{i_{2},j_{2}}:\ward_{\psi,\C}[[x]]\times\ward_{\psi,\C}[[x]]\rightarrow\ward_{\psi,\C}[[x]]$ by
\begin{equation}
    f\ast_{i_{1},j_{1}}\ast_{i_{2},j_{2}}g=\sum_{n=0}^{\infty}\left(\sum_{k=0}^{n}F(n+i_{1},k+j_{1})F(n+i_{2},k+j_{2})\binom{n}{k}_{\psi}a_{k}b_{n-k}\right)\frac{x^n}{\psi_{n}!}.
\end{equation}
In general 
\begin{equation}
    f\left(\prod_{h=1}^{l}\ast_{i_{h},j_{h}}\right)g=\sum_{n=0}^{\infty}\left(\sum_{k=0}^{n}\prod_{h=1}^{l} F(n+i_{h},k+j_{h})\binom{n}{k}_{\psi}a_{k}b_{n-k}\right)\frac{x^n}{\psi_{n}!}
\end{equation}
with $i_{h},j_{h}\in\N$, $j_{h}<i_{h}$ and $\prod_{h=1}^{l}\ast_{i_{h},j_{h}}=\ast_{i_{1},j_{1}}\cdots\ast_{i_{l},j_{l}}$.
\end{definition}

Set $\textbf{i}=(i_{1},...i_{l})$ and $\textbf{j}=(j_{1},...,j_{l})$. Now define the maps $M_{\textbf{i},\textbf{j}}:\ward_{\psi,\C}[[x]]\rightarrow\ward_{\psi,C}[[x]]$ by $M_{\textbf{i},\textbf{j}}=M_{i_{1},j_{1}}\cdots M_{i_{l},j_{l}}$ and $L_{\textbf{i},\textbf{j}}:\ward_{\psi,\C}[[x]]\rightarrow\ward_{\psi,C}[[x]]$ by $L_{\textbf{i},\textbf{j}}=L_{i_{1},j_{1}}\cdots L_{i_{l},j_{l}}$, where the product is by composition of maps. The proof of the following theorem is similar to the proof of the theorem \ref{theo_anillo_i_j} and will therefore be omitted.

\begin{theorem}\label{theo_anillo_prod_i_j}
If $\psi_{n}=n$, then $(\ward_{\psi,\C}[[x]],+,\prod_{h=1}^{l}\ast_{i_{h},j_{h}})$ is a commutative ring with unit. If $\psi_{n}\neq n$, then the set $(\ward_{\psi,\C}[[x]],+,\prod_{h=1}^{l}\ast_{i_{h},j_{h}})$ is a non-associative and non-commutative ring. For all $f(x)$ en $\ward_{\psi,\C}[[x]]$ it holds that 
\begin{equation*}
f(x)\left(\prod_{h=1}^{l}\ast_{i_{h},j_{h}}\right)e=eM_{\textbf{i},\textbf{j}}(f(x))    
\end{equation*}
and 
\begin{equation*}
e\left(\prod_{h=1}^{l}\ast_{i_{h},j_{h}}\right)f(x)=eL_{\textbf{i},\textbf{j}}(f(x))    
\end{equation*}
with $e=\frac{1}{\prod_{h=1}^{l}F(i_{h},j_{h})}$. When $\psi_{n}=[n]$, $e$ is an unit on left-hand side in $\ward_{q,\C}[[x]]$. In general, if $j_{h}=0$ for $1\leq h\leq l$, then $1$ is an unit on the left-hand side in $\ward_{\psi,\C}[[x]]$. We will call $(\ward_{\psi,\C}[[x]],+,\prod_{h=1}^{l}\ast_{i_{h},j_{h}})$ the $\psi$-ring of Ward-Fonten\'e with product $\prod_{h=1}^{l}\ast_{i_{h},j_{h}}$.
\end{theorem}

We conclude this section by constructing a $\C$-algebra of products defined on the set $\ward_{\psi,\C}[[x]]$. 

\begin{theorem}\label{theo_alg_prod}
Let $\C_{\psi}^{\ast}$ denote the set of all Fonten\'e products $\ast_{i,j}$, $j<i$, defined over the $\psi$-ring of Ward $\ward_{\psi,\C}[[x]]$ and define the following operations in $\C_{\psi}^\ast$: the formal sum $\boxplus$ as
\begin{equation}\label{eqn_sum}
f(\ast_{i_{1},j_{1}}\boxplus\ast_{i_{2},j_{2}})g=f\ast_{i_{1},j_{1}}g+f\ast_{i_{2},j_{2}}g,
\end{equation}
product $\cdot$ given by concatenation, i.e.
\begin{equation}\label{eqn_concat}
\ast_{i_{1},j_{1}}\cdot\ast_{i_{2},j_{2}}=\ast_{i_{1},j_{1}}\ast_{i_{2},j_{j_{2}}}    
\end{equation}
and scalar product
\begin{equation}\label{eqn_escalar}
f(\alpha\ast_{i,j})g=(\alpha f)\ast_{i,j}g=f\ast_{i,j}(\alpha g)=\alpha(f\ast_{i,j}g)
\end{equation}
for all $f,g\in\ward_{\psi,\C}[[x]]$, all $\alpha\in\C$ and all $\ast_{i_{1},j_{1}},\ast_{i_{2},j_{2}}\in\C_{\psi}^\ast$. Then $(\C_{\psi}^{\ast},\boxplus,\cdot)$ is a commutative unital algebra over $\C$.
\end{theorem}
\begin{proof}
From Eq.(\ref{eqn_escalar}) it follows that $f(0\ast_{i,j})g=0$. Now denote $\ast_{0}=0\ast_{i,j}$ to then show that $\ast_{0}$ is the neutral element in $\C_{\psi}^\ast$. This follows because
\begin{eqnarray*}
    f(\ast_{i,j}\boxplus\ast_{0})g&=&f\ast_{i,j}g+f\ast_{0}g\\
    &=&f\ast_{i,j}g+f(0\ast_{k,l})g\\
    &=&f\ast_{i,j}g+0\\
    &=&f\ast_{i,j}g
\end{eqnarray*}
and thus $\ast_{i,j}\boxplus\ast_{0}=\ast_{0}\boxplus\ast_{i,j}=\ast_{i,j}$. We define $\boxminus$ as $\ast_{i_{1},j_{1}}\boxminus\ast_{i_{2},j_{2}}=\ast_{i_{1},j_{1}}\boxplus(-\ast_{i_{2},j_{2}})$. Then
\begin{eqnarray*}
    f(\ast_{i,j}\boxminus\ast_{i,j})g&=&f\ast_{i,j}g-f\ast_{i,j}g=0=f\ast_{0}g
\end{eqnarray*}
for all $f,g\in\ward_{\psi,\C}[[x]]$. Then $\ast_{i,j}\boxminus\ast_{i,j}=\ast_{0}$ so $\ast_{i,j}$ and $-\ast_{i,j}$ are additive inverses in $\C_{\psi}^\ast$. It is easily shown that the sum $\boxplus$ is commutative and associative. Then the set $(\C_{\psi}^\ast,\boxplus)$ is an abelian group. Pick $\alpha,\beta$ in $\C$. Then
\begin{equation*}
    f(\alpha(\beta\ast_{i,j})g)=(\alpha f)(\beta\ast_{i,j}g)=((\alpha\beta)f)\ast_{i,j}g=f((\alpha\beta)\ast_{i,j})g
\end{equation*}
thus $\alpha(\beta\ast_{i,j})=(\alpha\beta)\ast_{i,j}$. As
\begin{align*}
    f((\alpha+\beta)\ast_{i,j})g&=((\alpha+\beta)f)\ast_{i,j}g\\
    &=(\alpha f+\beta g)\ast_{i,j}g\\
    &=(\alpha f)\ast_{i,j}g+(\beta f)\ast_{i,j}g\\
    &=f(\alpha\ast_{i,j})g+f(\beta\ast_{i,j})g\\
    &=f(\alpha\ast_{i,j}\boxplus\beta\ast_{i,j})g
\end{align*}
it follows then that $(\alpha+\beta)\ast_{i,j}=\alpha\ast_{i,j}\boxplus\beta\ast_{i,j}$. Similarly, it is shown that $\alpha(\ast_{i,j}\boxplus\ast_{k,l})=\alpha\ast_{i,j}\boxplus\alpha\ast_{k,l}$, for $\ast_{i,j},\ast_{k,l}\in\C_{\psi}^\ast$, $j<i$, $l<k$. On the other hand, it follows easily from (\ref{eqn_concat}) that
\begin{equation*}
(\ast_{i_{1},j_{1}}\ast_{i_{2},j_{2}})\ast_{i_{3},j_{3}}=\ast_{i_{1},j_{1}}(\ast_{i_{2},j_{2}}\ast_{i_{3},j_{3}})
\end{equation*}
and the concatenation is associative. In addition, it is easy to notice that the product $\ast_{\infty,\infty}$ is the unit in $\C_{\psi}^\ast$, because 
\begin{equation*}
f(\ast_{\infty,\infty}\ast_{i,j})g=f(\ast_{i,j}\ast_{\infty,\infty})g=f\ast_{i,j}g,
\end{equation*}
and from here $\ast_{\infty,\infty}\ast_{i,j}=\ast_{i,j}\ast_{\infty,\infty}=\ast_{i,j}$. The commutativity is evident. Finally we have that
\begin{align*}
\ast_{i,j}(\ast_{k,l}\boxplus\ast_{r,s})&=\ast_{i,j}\ast_{k,l}\boxplus\ast_{i,j}\ast_{r,s}\\    
(\ast_{k,l}\boxplus\ast_{r,s})\ast_{i,j}&=\ast_{k,l}\ast_{i,j}\boxplus\ast_{r,s}\ast_{i,j}\\
(\alpha\ast_{i,j})(\beta\ast_{k,l})&=(\alpha\beta)(\ast_{i,j}\ast_{k,l})
\end{align*}
for all $\ast_{i,j},\ast_{k,l},\ast_{r,s}\in\C_{\psi}^\ast$
and therefore $\C_{\psi}^\ast$ is a commutative unital algebra over $\C$.
\end{proof}

We will note that the structure of this $\C$-algebra depends on the nature of the sequence $\psi$. When $\psi_{n}=n$, then $\C_{\psi}^\ast=\{\ast_{\infty,\infty}\}$, i.e., the algebra $\C_{\psi}^\ast$ reduces to the ordinary product in $\ward_{\C}[[x]]$. When $\psi_{n}=[n]$ we have that $f(x)\ast_{i,j}g(x)=q^{j}f(qx)g(x)$ and it is noted that this product does not depend on $i$. Then
\begin{align*}
    f\ast_{1,0}g&=f\ast_{2,0}g=f\ast_{3,0}g=\cdots=f(qx)g(x)\\
    f\ast_{2,1}g&=f\ast_{3,1}g=f\ast_{4,1}g=\cdots=qf(qx)g(x)\\
    f\ast_{3,2}g&=f\ast_{4,2}g=f\ast_{5,2}g=\cdots=q^{2}f(qx)g(x)\\
                &\vdots   
\end{align*}
for all $f,g\in\ward_{q,\C}[[x]]$ and a fixed $j$. Where $(\ward_{q,\C}[[x]],+,\ast_{r,j})$ is isomorphic to $(\ward_{q,\C}[[x]],+,\ast_{s,j})$. Likewise
\begin{equation*}
    \left(\ward_{q,\C}[[x]],+,\prod_{i=1}^{l}\ast_{a_{r},b_{j}}\right)\simeq\left(\ward_{q,\C}[[x]],+,\prod_{i=1}^{l}\ast_{c_{s},d_{j}}\right)
\end{equation*}
provided that $b_{1}+\cdots+b_{l}=d_{1}+\cdots+d_{l}$.

\begin{definition}
A Fonten\'e polynomial over $\C$ in the product $\ast_{i,j}$ is an expression of the form
\begin{equation}
    a_{0}\ast_{\infty,\infty}\boxplus a_{1}\ast_{i,j}\boxplus\cdots\boxplus a_{n}\ast_{i,j}^n,
\end{equation}
where $\ast_{i,j}^{0}=\ast_{\infty,\infty}$, $\ast_{i,j}^{1}=\ast_{i,j}$ and $\ast_{i,j}^{n}=\ast_{i,j}\cdot\ast_{i,j}^{n-1}$, $n\geq1$.
The set of all polynomials in $\ast_{i,j}$ over $\C$ is denote by $\C[\ast_{i,j}]$.
\end{definition}
Some polynomials in $\C[\ast_{i,j}]$ are:
\begin{align*}
\ast_{i,j}^{2}\boxplus(2\ast_{i,j})\boxplus\ast_{\infty,\infty}&=(\ast_{i,j}\boxplus\ast_{\infty,\infty})(\ast_{i,j}\boxplus\ast_{\infty,\infty}),\\
\ast_{i,j}^{n}\boxminus\ast_{\infty,\infty}&=(\ast_{i,j}\boxminus\ast_{\infty,\infty})\bigboxplus_{k=0}^{n-1}\ast_{i,j}^{k}.
\end{align*}
Fonten\'e polynomials are added pointwise,
\begin{equation*}
    A\boxplus B=C\text{ when }c_{n}=a_{n}+b_{n}\text{ for all }n\geq0,
\end{equation*}
and multiplied by rule,
\begin{equation*}
    A\cdot B=C\text{ when }c_{n}=\sum_{i+j=n}a_{i}b_{j}\text{ for all }n\geq0
\end{equation*}
for all $A,B\in\C[\ast_{i,j}]$. It is then very straightforward to show the following result.
\begin{theorem}
$\C[\ast_{i,j}]$ is a commutative ring with unit $\ast_{\infty,\infty}$.
\end{theorem}

\begin{definition}
A Fonten\'e monomial in the set of products $\{\ast_{i_{1},j_{1}},\ast_{i_{2},j_{2}},...,\ast_{i_{r},j_{r}}\}$, $j_{a}<i_{a}$, is the product $\prod_{a=1}^{r}\ast_{i_{a},j_{a}}^{k_{i_{a},j_{a}}}$ with integer exponents $k_{i_{a},j_{a}}\geq0$.
\end{definition}

\begin{definition}
A Fonten\'e polynomial in $n$ variables over $\C$ in the monomial $\prod_{a=1}^{r}\ast_{i_{a},j_{a}}^{k_{i_{a},j_{a}}}$ is an expression of the form
\begin{equation}
    a_{0}\ast_{\infty,\infty}\boxplus a_{1}\left(\prod_{a=1}^{r}\ast_{i_{a},j_{a}}^{k_{i_{a},j_{a}}}\right)\boxplus\cdots\boxplus a_{n}\left(\prod_{a=1}^{r}\ast_{i_{a},j_{a}}^{k_{i_{a},j_{a}}}\right)^n.
\end{equation}
The set of all Fonten\'e polynomials in $n$-variables over $\C$ is denote by $\C[\{\ast_{i_{a},j_{a}}\}_{a=1}^{r}]$.
\end{definition}

The following result follows from Theorem \ref{theo_anillo_prod_i_j}.
\begin{theorem}
Let $p(\ast_{i_{1},j_{1}},...,\ast_{i_{r},j_{r}})$ denote a Fonten\'e polynomial in $\C[\{\ast_{i_{a},j_{a}}\}_{a=1}^{r}]$. Then $(\ward_{\psi,\C}[[x]],+,p(\ast_{i_{1},j_{1}},...,\ast_{i_{r},j_{r}}))$ is a non-commutative non-associative ring.
\end{theorem}

\begin{definition}
Let $p(\ast_{i_{1},j_{1}},...,\ast_{i_{r},j_{r}})$ denote a Fonten\'e polynomial in $\C[\{\ast_{i_{a},j_{a}}\}_{a=1}^{r}]$. Define the linear maps
\begin{align*}
\rho&:\C[\ast_{i_{1},j_{1}},...,\ast_{i_{r},j_{r}}]\rightarrow\C[\ast_{i_{1}+1,j_{1}+1},...,\ast_{i_{r}+1,j_{r}+1}],\\
\sigma&:\C[\ast_{i_{1},j_{1}},...,\ast_{i_{r},j_{r}}]\rightarrow\C[\ast_{i_{1}+1,j_{1}},...,\ast_{i_{r}+1,j_{r}},\ast_{1,0}]
\end{align*}
by
\begin{align*}
\rho(p(\ast_{i_{1},j_{1}},...,\ast_{i_{r},j_{r}}))&=\rho\left(\bigboxplus_{h=0}^{n}a_{h}\left(\prod_{a=1}^{r}\ast_{i_{a},j_{a}}^{k_{i_{a},j_{a}}}\right)^h\right)=\bigboxplus_{h=0}^{n}a_{h}\left(\prod_{a=1}^{r}\ast_{i_{a}+1,j_{a}+1}^{k_{i_{a},j_{a}}}\right)^h,\\
\sigma(p(\ast_{i_{1},j_{1}},...,\ast_{i_{r},j_{r}}))&=\sigma\left(\bigboxplus_{h=0}^{n}a_{h}\left(\prod_{a=1}^{r}\ast_{i_{a},j_{a}}^{k_{i_{a},j_{a}}}\right)^h\right)=\bigboxplus_{h=0}^{n}a_{h}\left(\prod_{a=1}^{r}\ast_{i_{a}+1,j_{a}}^{k_{i_{a},j_{a}}}\ast_{1,0}\right)^h
\end{align*}
and
$\rho(\ast_{\infty,\infty})=\ast_{\infty,\infty}$ and $\sigma(\ast_{\infty,\infty})=\ast_{1,0}$.
\end{definition}

\begin{proposition}
The map $\rho$ is a homomorphims of rings.
\end{proposition}
\begin{proof}
Follows directly from the definition of $\rho$.
\end{proof}

\begin{definition}
Let $p(\ast_{i_{1},j_{1}},...,\ast_{i_{r},j_{r}})$ denote a Fonten\'e polynomial in $\C[\{\ast_{i_{a},j_{a}}\}_{a=1}^{r}]$. Define the $p(\ast_{i_{1},j_{1}},...,\ast_{i_{r},j_{r}})$-universal algebra of Ward-Fonten\'e as
\begin{equation}
    \ward_{\psi,\C}[[x]]_{p(\ast_{i_{1},j_{1}},...,\ast_{i_{r},j_{r}})}=\left(\ward_{\psi,\C}[[x]],+,\bigcup_{n=0}^{\infty}\bigcup_{X\in\{\rho,\sigma\}^n}X(p(\ast_{i_{1},j_{1}},...,\ast_{i_{r},j_{r}}))\right)
\end{equation}
where $\{\rho,\sigma\}^n$ denote the $n$-fold composition of $\{\rho,\sigma\}$ defined by $\{\rho,\sigma\}^{0}=\{1\}$, the identity map, $\{\rho,\sigma\}^{1}=\{\rho,\sigma\}$ and $\{\rho,\sigma\}^{n}=\rho\{\rho,\sigma\}^{n-1}\cup\sigma\{\rho,\sigma\}^{n-1}$ for $n\geq1$. Also, define the $\ast_{\infty,\infty}$-universal algebra of Ward-Fonten\'e as
\begin{equation}
    \ward_{\psi,\C}[[x]]_{\ast_{\infty,\infty}}=\left(\ward_{\psi,\C}[[x]],+,\bigcup_{n=0}^{\infty}\bigcup_{X\in\{\rho,\sigma\}^n}X(\ast_{\infty,\infty})\right).
\end{equation}
\end{definition}

\section{Ward-Fonten\'e opposite universal algebras}

In this section we will construct the opposite $\psi$-rings of Ward-Fonten\'e, where this time we will use the equation (\ref{eqn_psi_binom_sim}) for the definition of the Fonten\'e products.

\begin{definition}
For all $i,j$ with $j<i$ and for all $f,g\in\ward_{\psi,\C}[[x]]$ we define the product $\star_{i,j}:\ward_{\psi,\C}[[x]]\times\ward_{\psi,\C}[[x]]\rightarrow\ward_{\psi,\C}[[x]]$ as
\begin{equation}
    f\star_{i,j}g=\sum_{n=0}^{\infty}\sum_{k=0}^{n}F(n+i,n-k+j)\binom{n}{k}_{\psi}a_{k}b_{n-k}\frac{x^n}{\psi_{n}!}.
\end{equation}
We will call $\star_{i,j}$ the opposite Fonten\'e product
\end{definition}

\begin{definition}
For all $i_{1},i_{2},j_{1},j_{2}\in\N$ with $j_{1}<i_{1}$, $j_{2}<i_{2}$ and for all $f,g\in\ward_{\psi,\C}[[x]]$ we define the product $\star_{i_{1},j_{1}}\star_{i_{2},j_{2}}:\ward_{\psi,\C}[[x]]\times\ward_{\psi,\C}[[x]]\rightarrow\ward_{\psi,\C}[[x]]$ as
\begin{equation}
    f\star_{i_{1},j_{1}}\star_{i_{2},j_{2}}g=\sum_{n=0}^{\infty}\sum_{k=0}^{n}F(n+i_{1},n-k+j_{1})F(n+i_{2},n-k+j_{2})\binom{n}{k}_{\psi}a_{k}b_{n-k}\frac{x^n}{\psi_{n}!}.
\end{equation}
In general, we define by concatenation the product $\prod_{i=1}^{l}\star_{a_{i},b_{i}}$ as
\begin{equation}
    f\left(\prod_{i=1}^{l}\star_{a_{i},b_{i}}\right)g=\sum_{n=0}^{\infty}\sum_{k=0}^{n}\prod_{i=1}^{l} F(n+a_{i},n-k+b_{i})\binom{n}{k}_{\psi}a_{k}b_{n-k}\frac{x^n}{\psi_{n}!}
\end{equation}
with $a_{i},b_{i}\in\N$, $b_{i}<a_{i}$ and $\prod_{i=1}^{l}\star_{a_{i},b_{i}}=\star_{a_{1},b_{1}}\cdots\star_{a_{l},b_{l}}$.
\end{definition}

The following are the symmetric versions of the theorems \ref{theo_anillo_i_j}, \ref{theo_anillo_prod_i_j} and \ref{theo_alg_prod}

\begin{theorem}
If $\psi_{n}=n$, then $(\ward_{\psi,\C}[[x]],+,\star_{i,j})$ is a commutative ring with unit. If $\psi_{n}\neq n$, then the set $(\ward_{\psi,\C}[[x]],+,\star_{i,j})$ is a non-associative and non-commutative ring. For all $f(x)$ en $\ward_{\psi,\C}[[x]]$ it holds that $f(x)\star_{i,j}e=eL_{i,j}(f(x))$ and $e\star_{i,j}f(x)=eM_{i,j}(f(x))$ with $e=\frac{1}{F(i,j)}$. When $\psi_{n}=[n]$, $e$ is an unit on right-hand side in $\ward_{q,\C}[[x]]$. In general, if $j=0$, then $1$ is an unit on the right-hand side in $\ward_{\psi,\C}[[x]]$. We will call $(\ward_{\psi,\C}[[x]],+,\ast_{i,j})$ the opposite $\psi$-ring of Ward-Fonten\'e with product $\star_{i,j}$.
\end{theorem}

\begin{theorem}\label{theo_ring_opp_prod_i_j}
If $\psi_{n}=n$, then $(\ward_{\psi,\C}[[x]],+,\prod_{h=1}^{l}\star_{i_{h},j_{h}})$ is a commutative ring with unit. If $\psi_{n}\neq n$, then the set $(\ward_{\psi,\C}[[x]],+,\prod_{h=1}^{l}\star_{i_{h},j_{h}})$ is a non-associative and non-commutative ring. For all $f(x)$ en $\ward_{\psi,\C}[[x]]$ it holds that 
\begin{equation*}
f(x)\left(\prod_{h=1}^{l}\star_{i_{h},j_{h}}\right)e=eL_{\textbf{i},\textbf{j}}(f(x))    
\end{equation*}
and 
\begin{equation*}
e\left(\prod_{h=1}^{l}\star_{i_{h},j_{h}}\right)f(x)=eM_{\textbf{i},\textbf{j}}(f(x))    
\end{equation*}
with $e=\frac{1}{\prod_{h=1}^{l}F(i_{h},j_{h})}$. When $\psi_{n}=[n]$, $e$ is an unit on right-hand side in $\ward_{q,\C}[[x]]$. In general, if $j_{h}=0$ for $1\leq h\leq l$, then $1$ is an unit on the right-hand side in $\ward_{\psi,\C}[[x]]$. We will call $(\ward_{\psi,\C}[[x]],+,\prod_{h=1}^{l}\star_{i_{h},j_{h}})$ the opposite $\psi$-ring of Ward-Fonten\'e with product $\prod_{i=1}^{l}\star_{a_{i},b_{i}}$.
\end{theorem}

\begin{theorem}
Let $\C_{\psi}^{\star}$ denote the set of all Fonten\'e products $\star_{i,j}$, $j<i$, defined over the $\psi$-ring of Ward $\ward_{\psi,\C}[[x]]$ and define the following operations in $\C_{\psi}^\star$: the formal sum $\boxplus$ as
\begin{equation}
f(\star_{i_{1},j_{1}}\boxplus\star_{i_{2},j_{2}})g=f\star_{i_{1},j_{1}}g+f\star_{i_{2},j_{2}}g,
\end{equation}
product $\cdot$ given by concatenation, i.e.
\begin{equation}
\star_{i_{1},j_{1}}\cdot\star_{i_{2},j_{2}}=\star_{i_{1},j_{1}}\star_{i_{2},j_{j_{2}}}    
\end{equation}
and scalar product
\begin{equation}
f(\alpha\star_{i,j})g=(\alpha f)\star_{i,j}g=f\star_{i,j}(\alpha g)=\alpha(f\star_{i,j}g)
\end{equation}
for all $f,g\in\ward_{\psi,\C}[[x]]$, all $\alpha\in\C$ and all $\star_{i_{1},j_{1}},\star_{i_{2},j_{2}}\in\C_{\psi}^\star$. Then $(\C_{\psi}^{\star},\boxplus,\cdot)$ is a commutative unital algebra over $\C$.
\end{theorem}

It is clear that $f\prod_{i=1}^{l}\ast_{a_{i},b_{i}}g=g\prod_{i=1}^{l}\star_{a_{i},b_{i}}f$ for all $f,g\in\ward_{\psi,\C}[[x]]$. Then the rings $(\ward_{\psi,\C}[[x]],+,\prod_{i=1}^{l}\star_{a_{i},b_{i}})$ and $(\ward_{\psi,\C}[[x]],+,\prod_{i=1}^{l}\ast_{a_{i},b_{i}})$
are opposite rings and therefore isomorphic.

\begin{definition}
A Fonten\'e polynomial over $\C$ in the product $\star_{i,j}$ is an expression of the form
\begin{equation}
    a_{0}\star_{\infty,\infty}\boxplus a_{1}\star_{i,j}\boxplus\cdots\boxplus a_{n}\star_{i,j}^n,
\end{equation}
where $\star_{i,j}^{0}=\star_{\infty,\infty}$, $\star_{i,j}^{1}=\star_{i,j}$ and $\star_{i,j}^{n}=\star_{i,j}\cdot\star_{i,j}^{n-1}$, $n\geq1$.
The set of all polynomials in $\star_{i,j}$ over $\C$ is denote by $\C[\star_{i,j}]$.
\end{definition}

Fonten\'e polynomials are added pointwise,
\begin{equation*}
    A\boxplus B=C\text{ when }c_{n}=a_{n}+b_{n}\text{ for all }n\geq0,
\end{equation*}
and multiplied by rule,
\begin{equation*}
    A\cdot B=C\text{ when }c_{n}=\sum_{i+j=n}a_{i}b_{j}\text{ for all }n\geq0
\end{equation*}
for all $A,B\in\C[\star_{i,j}]$. It is then very easy to show the following result.
\begin{theorem}
$\C[\star_{i,j}]$ is a commutative ring with unit $\star_{\infty,\infty}$.
\end{theorem}

\begin{definition}
A Fonten\'e monomial in the set of products $\{\star_{i_{1},j_{1}},\star_{i_{2},j_{2}},...,\star_{i_{r},j_{r}}\}$, $j_{a}<i_{a}$, is the product $\prod_{a=1}^{r}\star_{i_{a},j_{a}}^{k_{i_{a},j_{a}}}$ with integer exponents $k_{i_{a},j_{a}}\geq0$.
\end{definition}

\begin{definition}
A Fonten\'e polynomial in $n$ variables over $\C$ in the monomial $\prod_{a=1}^{r}\star_{i_{a},j_{a}}^{k_{i_{a},j_{a}}}$ is an expression of the form
\begin{equation}
    a_{0}\star_{\infty,\infty}\boxplus a_{1}\left(\prod_{a=1}^{r}\star_{i_{a},j_{a}}^{k_{i_{a},j_{a}}}\right)\boxplus\cdots\boxplus a_{n}\left(\prod_{a=1}^{r}\star_{i_{a},j_{a}}^{k_{i_{a},j_{a}}}\right)^n.
\end{equation}
The set of all Fonten\'e polynomials in $n$-variables over $\C$ is denote by $\C[\{\star_{i_{a},j_{a}}\}_{a=1}^{r}]$.
\end{definition}

The following result follows from Theorem \ref{theo_ring_opp_prod_i_j}.
\begin{theorem}
Let $p(\star_{i_{1},j_{1}},...,\ast_{i_{r},j_{r}})$ denote a Fonten\'e polynomial in $\C[\{\star_{i_{a},j_{a}}\}_{a=1}^{r}]$. Then $(\ward_{\psi,\C}[[x]],+,p(\star_{i_{1},j_{1}},...,\star_{i_{r},j_{r}}))$ is a non-commutative non-associative ring.
\end{theorem}

\begin{definition}
Let $p(\star_{i_{1},j_{1}},...,\star_{i_{r},j_{r}})$ denote a Fonten\'e polynomial in $\C[\{\star_{i_{a},j_{a}}\}_{a=1}^{r}]$. Define the linear maps
\begin{align*}
\rho&:\C[\star_{i_{1},j_{1}},...,\star_{i_{r},j_{r}}]\rightarrow\C[\star_{i_{1}+1,j_{1}+1},...,\star_{i_{r}+1,j_{r}+1}],\\
\sigma&:\C[\star_{i_{1},j_{1}},...,\star_{i_{r},j_{r}}]\rightarrow\C[\star_{i_{1}+1,j_{1}},...,\star_{i_{r}+1,j_{r}},\star_{1,0}]
\end{align*}
by
\begin{align*}
\rho(p(\star_{i_{1},j_{1}},...,\star_{i_{r},j_{r}}))&=\rho\left(\bigboxplus_{h=0}^{n}a_{h}\left(\prod_{a=1}^{r}\star_{i_{a},j_{a}}^{k_{i_{a},j_{a}}}\right)^h\right)=\bigboxplus_{h=0}^{n}a_{h}\left(\prod_{a=1}^{r}\star_{i_{a}+1,j_{a}+1}^{k_{i_{a},j_{a}}}\right)^h,\\
\sigma(p(\star_{i_{1},j_{1}},...,\star_{i_{r},j_{r}}))&=\sigma\left(\bigboxplus_{h=0}^{n}a_{h}\left(\prod_{a=1}^{r}\star_{i_{a},j_{a}}^{k_{i_{a},j_{a}}}\right)^h\right)=\bigboxplus_{h=0}^{n}a_{h}\left(\prod_{a=1}^{r}\star_{i_{a}+1,j_{a}}^{k_{i_{a},j_{a}}}\star_{1,0}\right)^h
\end{align*}
and
$\rho(\star_{\infty,\infty})=\star_{\infty,\infty}$ and $\sigma(\star_{\infty,\infty})=\star_{1,0}$.
\end{definition}

\begin{proposition}
The map $\rho$ is a homomorphims of rings.
\end{proposition}
\begin{proof}
Follows directly from the definition of $\rho$.
\end{proof}

\begin{definition}
Let $p(\star_{i_{1},j_{1}},...,\star_{i_{r},j_{r}})$ denote a Fonten\'e polynomial in $\C[\{\star_{i_{a},j_{a}}\}_{a=1}^{r}]$. Define the $p(\star_{i_{1},j_{1}},...,\star_{i_{r},j_{r}})$-universal algebra of Ward-Fonten\'e as
\begin{equation}
    \ward_{\psi,\C}[[x]]_{p(\star_{i_{1},j_{1}},...,\star_{i_{r},j_{r}})}=\left(\ward_{\psi,\C}[[x]],+,\bigcup_{n=0}^{\infty}\bigcup_{X\in\{\rho,\sigma\}^n}X(p(\star_{i_{1},j_{1}},...,\star_{i_{r},j_{r}}))\right)
\end{equation}
where $\{\rho,\sigma\}^n$ denote the $n$-fold composition of $\{\rho,\sigma\}$ defined by $\{\rho,\sigma\}^{0}=\{1\}$, the identity map, $\{\rho,\sigma\}^{1}=\{\rho,\sigma\}$ and $\{\rho,\sigma\}^{n}=\rho\{\rho,\sigma\}^{n-1}\cup\sigma\{\rho,\sigma\}^{n-1}$ for $n\geq1$. Also, define the $\star_{\infty,\infty}$-universal algebra of Ward-Fonten\'e as
\begin{equation}
    \ward_{\psi,\C}[[x]]_{\star_{\infty,\infty}}=\left(\ward_{\psi,\C}[[x]],+,\bigcup_{n=0}^{\infty}\bigcup_{X\in\{\rho,\sigma\}^n}X(\star_{\infty,\infty})\right).
\end{equation}
\end{definition}

\section{$\psi$-Derivative of the Fonten\'e products in $\ward_{\psi,\C}[[x]]$}

In this section we will find the $\psi$-derivative of the Fonten\'e products of $\psi$-exponential generating functions in $\ward_{\psi,\C}[[x]]$. In particular we will compute the $\psi$-derivative of the ordinary product of functions from which the ordinary case and the $q$-analogue are deduced.

\begin{theorem}
Take $f,g\in\ward_{\psi,\C}[[x]]$. Then the $\psi$-derivative of the product $f\ast_{i,j}g$ is given by
\begin{equation}\label{eqn_der_ij}
    \mathbf{D}_{\psi}(f\ast_{i,j}g)=\mathbf{D}_{\psi}f(\rho(\ast_{i,j})) g+f(\sigma(\ast_{i,j}))\mathbf{D}_{\psi}g.
\end{equation}
\end{theorem}
\begin{proof}
Set $f(x)=\sum_{n=0}^{\infty}a_{n}(t^n/\psi_{n}!)$ and $g(x)=\sum_{n=0}^{\infty}b_{n}(t^n/\psi_{n}!)$. Then
\begin{eqnarray*}
\mathbf{D}_{\psi}(f(x)\ast_{i,j}g(x))&=&\mathbf{D}_{\psi}\left(\sum_{n=0}^{\infty}\sum_{k=0}^{n}F(n+i,k+j)\binom{n}{k}_{\psi}a_{k}b_{n-k}\frac{x^n}{\psi_{n}!}\right)\\
&=&\sum_{n=0}^{\infty}\sum_{k=0}^{n+1}F(n+i+1,k+j)\binom{n+1}{k}_{\psi}a_{k}b_{n+1-k}\frac{x^n}{\psi_{n}!}.
\end{eqnarray*}
We take series for $k=0$ and for $k=n+1$ from above sum 
\begin{align*}
\mathbf{D}_{\psi}(f(x)\ast_{i,j}g(x))
&=\sum_{n=0}^{\infty}F(n+i+1,j)a_{0}b_{n+1}\frac{x^n}{\psi_{n}!}\\
&\qquad+\sum_{n=0}^{\infty}\sum_{k=1}^{n}F(n+i+1,k+j)\binom{n+1}{k}_{\psi}a_{k}b_{n+1-k}\frac{x^n}{\psi_{n}!}\\
&\qquad\qquad+\sum_{n=0}^{\infty}F(n+i+1,n+1+j)a_{n+1}b_{0}\frac{x^n}{\psi_{n}!}.
\end{align*}
Now we will use the Eq.(\ref{eqn_psi_binom}) to separate the sum of the middle into two summands
\begin{multline*}
\mathbf{D}_{\psi}(f(x)\ast_{i,j}g(x))
=\sum_{n=0}^{\infty}F(n+i+1,j)a_{0}b_{n+1}\frac{x^n}{\psi_{n}!}+\\
\sum_{n=0}^{\infty}\sum_{k=1}^{n}F(n+i+1,k+j)
\left[\binom{n}{k-1}_{\psi}+F(n+1,k)\binom{n}{k}_{\psi}\right]a_{k}b_{n+1-k}\frac{t^n}{\psi_{n}!}\\
+\sum_{n=0}^{\infty}F(n+i+1,n+1+j)a_{n+1}b_{0}\frac{x^n}{\psi_{n}!}.
\end{multline*}
Then
\begin{align*}
\mathbf{D}_{\psi}(f(x)\ast_{i,j}g(x))
&=\sum_{n=0}^{\infty}F(n+i+1,j)a_{0}b_{n+1}\frac{x^n}{\psi_{n}!}\\
&+\sum_{n=0}^{\infty}\sum_{k=1}^{n}F(n+i+1,k+j)\binom{n}{k-1}_{\psi}a_{k}b_{n+1-k}\frac{t^n}{\psi_{n}!}\\
&+\sum_{n=0}^{\infty}\sum_{k=1}^{n}F(n+i+1,k+j)F(n+1,k)\binom{n}{k}_{\psi}a_{k}b_{n+1-k}\frac{t^n}{\psi_{n}!}\\
&\hspace{2em}+\sum_{n=0}^{\infty}F(n+i+1,n+1+j)a_{n+1}b_{0}\frac{x^n}{\psi_{n}!}.
\end{align*}
Now we rearrange the second sum for $k$
\begin{align*}
\mathbf{D}_{\psi}(f(x)\ast_{i,j}g(x))
&=\sum_{n=0}^{\infty}F(n+i+1,j)a_{0}b_{n+1}\frac{x^n}{\psi_{n}!}\\
&+\sum_{n=0}^{\infty}\sum_{k=0}^{n-1}F(n+i+1,k+j+1)\binom{n}{k}_{\psi}a_{k+1}b_{n-k}\frac{t^n}{\psi_{n}!}\\
&+\sum_{n=0}^{\infty}\sum_{k=1}^{n}F(n+i+1,k+j)F(n+1,k)\binom{n}{k}_{\psi}a_{k}b_{n+1-k}\frac{t^n}{\psi_{n}!}\\
&+\sum_{n=0}^{\infty}F(n+i+1,n+1+j)a_{n+1}b_{0}\frac{x^n}{\psi_{n}!}.
\end{align*}
Finally, we join the first sum with the third sum and the second sum with the fourth sum
\begin{align*}
\mathbf{D}_{\psi}(f(x)\ast_{i,j}g(x))
&=\sum_{n=0}^{\infty}\sum_{k=0}^{n}F(n+i+1,k+j+1)\binom{n}{k}_{\psi}a_{k+1}b_{n-k}\frac{t^n}{\psi_{n}!}\\
&\qquad\sum_{n=0}^{\infty}\sum_{k=0}^{n}F(n+i+1,k+j)F(n+1,k)\binom{n}{k}_{\psi}a_{k}b_{n+1-k}\frac{t^n}{\psi_{n}!}\\
&=\mathbf{D}_{\psi}f(x)\ast_{i+1,j+1}g(x)+f(x)\ast_{i+1,j}\ast_{1,0}\mathbf{D}_{\psi}g(x).
\end{align*}
In this way we reach the desired result.
\end{proof}

For the opposite product $\star_{i,j}$ we obtain the following result
\begin{theorem}
Tome $f,g\in\ward_{\psi,\C}[[x]]$. Then the $\psi$-derivative of the product $f\star_{i,j}g$ is given by
\begin{equation}
    \mathbf{D}_{\psi}(f\star_{i,j}g)=\mathbf{D}_{\psi}f(\sigma(\star_{i,j}))g+f(\rho(\star_{i,j}))\mathbf{D}_{\psi}g.
\end{equation}
\end{theorem}

Now we will obtain the $\psi$-derivative of the ordinary product of functions
\begin{theorem}\label{theo_prod_leibniz}
Take $f,g$ in $\ward_{\psi,\C}[[x]]$.Then the $\psi$-derivative $\mathbf{D}_{\psi}$ of the product $fg$ is
\begin{eqnarray}
    \mathbf{D}_{\psi}(f(x)g(x))&=&f(x)\ast_{1,0}\mathbf{D}_{\psi}g(x)+\mathbf{D}_{\psi}f(x)\cdot g(x)\\
    &=&\mathbf{D}_{\psi}f(x)\star_{1,0}g(x)+f(x)\mathbf{D}_{\psi}g(x).
\end{eqnarray}
\end{theorem}
\begin{proof}
We use a proof similar to the proof of the Eq.(\ref{eqn_der_ij}).
\end{proof}

The $\psi$-derivative of the product of functions is well defined as we can notice by making $g(x)=1$ and then interchanging $f$ with $g$.
\begin{equation*}
    \mathbf{D}_{\psi}(f(x))=\mathbf{D}_{\psi}(f(x)\cdot1)=f(x)\ast_{1,0}0+\mathbf{D}_{\psi}f(x)=\mathbf{D}_{\psi}f(x)
\end{equation*}
and
\begin{equation*}
    \mathbf{D}_{\psi}(f(x))=\mathbf{D}_{\psi}(1\cdot f(x))=1\ast_{1,0}\mathbf{D}_{\psi}f(x)=\mathbf{D}_{\psi}f(x).
\end{equation*}

On the other hand, when $\psi_{n}=n$, then we obtain Leibniz's ordinary rule for the product of functions
\begin{equation*}
    \mathbf{D}(f(x)g(x))=(\mathbf{D}f)(x)g(x)+f(x)(\mathbf{D}g)(x).
\end{equation*}
and when $\psi_{n}=[n]$, then we obtain the $q$-rule of Leibniz
\begin{eqnarray*}
    \mathbf{D}_{q}(f(x)g(x))&=&(\mathbf{D}_{q}f)(x)g(x)+f(x)\ast_{1,0}(\mathbf{D}_{q}g)(x)\\
    &=&(\mathbf{D}_{q}f)(x)g(x)+f(qx)(\mathbf{D}_{q}g)(x).
\end{eqnarray*}
and by symmetry
\begin{eqnarray*}
    \mathbf{D}_{q}(f(x)g(x))&=&f(x)(\mathbf{D}_{q}g)(x)+(\mathbf{D}_{q}f)(x)\star_{1,0}g(x)\\
    &=&f(x)(\mathbf{D}_{q}g)(x)+(\mathbf{D}_{q}f)(x)g(qx)
\end{eqnarray*}
thus reaching the two known results for the $q$-derivative of the product of functions. When $\psi_{n}=[n]_{p,q}$, we obtain another $(p,q)$-analogue of Leibniz's rule
\begin{eqnarray*}
    \mathbf{D}_{p,q}(f(x)g(x))&=&(\mathbf{D}_{p,q}f)(x)g(x)+f(x)\ast_{1,0}(\mathbf{D}_{p,q}g)(x)\\
    &=&f(x)(\mathbf{D}_{p,q}g)(x)+(\mathbf{D}_{p,q}f)(x)\star_{1,0}g(x)
\end{eqnarray*}
where
\begin{align*}
    f(x)\ast_{1,0}(\mathbf{D}_{p,q}g)(x)&=f(qx)(\mathbf{D}_{p,q}g)(x)+g(px)(\mathbf{D}_{p,q}f)(x)-(\mathbf{D}_{p,q}f)(x)g(x),\\
    (\mathbf{D}_{p,q}f)(x)\star_{1,0}g(x)&=f(px)(\mathbf{D}_{p,q}g)(x)+g(qx)(\mathbf{D}_{p,q}f)(x)-f(x)(\mathbf{D}_{p,q}g)(x)
\end{align*}
When $\psi_{n}=T_{n}$, then
\begin{align*}
    \mathbf{D}_{T}(f(x)g(x))&=(\mathbf{D}_{T}f)(x)g(x)+f(x)\ast_{1,0}(\mathbf{D}_{T}g)(x),\\
    &=f(x)(\mathbf{D}_{T}g)(x)+(\mathbf{D}_{T}f)(x)\star_{1,0}g(x)
\end{align*}
where $f\ast_{1,0}(\mathbf{D}_{T}g)(x)=f(x)(\mathbf{D}_{T}g)(x)+xf^{\prime}(x)g^{\prime}(x)$ and $(\mathbf{D}_{T}f)(x)\star_{1,0}g(x)=(\mathbf{D}_{T}f)(x)g(x)+xf^{\prime}(x)g^{\prime}(x)$. Finally, if $\psi_{n}=\mathrm{B}_{n}$, then
\begin{align*}
    \mathbf{D}_{B}(f(x)g(x))&=(\mathbf{D}_{B}f)(x)g(x)+f(x)\ast_{1,0}(\mathbf{D}_{B}g)(x),\\
    &=f(x)(\mathbf{D}_{B}g)(x)+(\mathbf{D}_{B}f)(x)\star_{1,0}g(x)
\end{align*}
where
\begin{align*}
    f(x)\ast_{1,0}(\mathbf{D}_{B}g)(x)&=\frac{1}{e}\sum_{k=0}^{\infty}f(kx)\frac{(\mathbf{D}_{k,0}g)(x)}{k!}+g(0)(\mathbf{D}_{B}f)(x)-(\mathbf{D}_{B}f)(x)g(x),\\
    (\mathbf{D}_{B}f)(x)\star_{1,0}g(x)&=\frac{1}{e}\sum_{k=0}^{\infty}g(kx)\frac{(\mathbf{D}_{k,0}f)(x)}{k!}+f(0)(\mathbf{D}_{B}g)(x)-f(x)(\mathbf{D}_{B}g)(x).
\end{align*}

\begin{example}
We will use the Theorem \ref{theo_prod_leibniz} to show that $\mathbf{D}_{\psi}(x^{n+m})=\psi_{n+m}x^{n+m-1}$. On the one hand, using the product $\ast_{1,0}$ we have
\begin{eqnarray*}
    \mathbf{D}_{\psi}(x^{n+m})&=&\mathbf{D}_{\psi}(x^{n}x^{m})\\
    &=&x^{n}\ast_{1,0}\psi_{m}x^{m-1}+\psi_{n}x^{n+m-1}\\
    &=&(\psi_{m}F(n+m,n)+\psi_{n})x^{n+m-1}\\
    &=&\left(\psi_{m}\frac{\psi_{n+m}-\psi_{n}}{\psi_{m}}+\psi_{n}\right)x^{n+m-1}\\
    &=&\psi_{n+m}x^{n+m-1}.
\end{eqnarray*}
On the other hand, with the product $\star_{1,0}$ we obtain
\begin{eqnarray*}
    \mathbf{D}_{\psi}(x^{n+m})&=&\mathbf{D}_{\psi}(x^{n}x^{m})\\
    &=&\psi_{n}x^{n-1}\star_{1,0}x^{m}+\psi_{m}x^{n+m-1}\\
    &=&(\psi_{n}F(n+m,m)+\psi_{m})x^{n+m-1}\\
    &=&\left(\psi_{n}\frac{\psi_{n+m}-\psi_{m}}{\psi_{n}}+\psi_{m}\right)x^{n+m-1}\\
    &=&\psi_{n+m}x^{n+m-1}.
\end{eqnarray*}
\end{example}

Now we will give Leibniz's rule for the products $\prod_{i=1}^{l}\ast_{a_{i},b_{i}}$ and $\prod_{i,j=1}^{l}\star_{a_{i},b_{j}}$ of functions.
\begin{theorem}\label{theo_der_prod_mu_}
Take $f,g$ in $\ward_{\psi,\C}[[x]]$. Then
\begin{align*}    \mathbf{D}_{\psi}\left(f\prod_{i=1}^{l}\ast_{a_{i},b_{i}}g\right)&=(\mathbf{D}_{\psi}f)\rho\left(\prod_{i=1}^{l}\ast_{a_{i},b_{i}}\right)g+f\sigma\left(\prod_{i=1}^{l}\ast_{a_{i},b_{i}}\right)\mathbf{D}_{\psi}g,\\
\mathbf{D}_{\psi}\left(f\prod_{i=1}^{l}\star_{a_{i},b_{i}}g\right)
&=f\rho\left(\prod_{i=1}^{l}\star_{a_{i},b_{i}}\right)\mathbf{D}_{\psi}g+(\mathbf{D}_{\psi}f)\sigma\left(\prod_{i=1}^{l}\star_{a_{i},b_{i}}\right)g.
\end{align*}
\end{theorem}
\begin{proof}
The proof is analogous to that given for the Eq.(\ref{eqn_der_ij}).
\end{proof}
When $\psi_{n}=[n]$ we get the following result
\begin{align*}
&\mathbf{D}_{q}\left(f(x)\left(\prod_{i=1}^{l}\ast_{a_{i},b_{i}}\right)g(x)\right)\\
&\hspace{5em}=\mathbf{D}_{q}\left(q^{b_{1}+\cdots+b_{l}}f(q^{l}x)g(x)\right)\\
&\hspace{5em}=q^{b_{1}+\cdots+b_{l}+l}\mathbf{D}_{q}f(q^{l}x)\cdot g(x)+q^{b_{1}+\cdots+b_{l}}f(q^{l+1}x)\mathbf{D}_{q}g(x).
\end{align*}
Then
\begin{equation*}
    \mathbf{D}_{q}(f(q^{l}x)g(x))=q^{l}\mathbf{D}_{q}f(q^{l}x)\cdot g(x)+f(q^{l}x)\mathbf{D}_{q}g(x).
\end{equation*}
If we make $g(x)=1$, we get
\begin{equation*}
    \mathbf{D}_{q}f(q^{l}x)=q^{l}(\mathbf{D}_{q}f)(q^{l}x)
\end{equation*}
which is the chain rule for the function $f(q^{l}x)$.

We ended with the following results.
\begin{theorem}
Take $p(\ast_{i_{1},j_{1}},...,\ast_{i_{r},j_{r}})\in\C[\{\ast_{i_{a},j_{a}}\}_{a=1}^r]$. Then
\begin{align*}
\mathbf{D}_{\psi}(f(p(\ast_{i_{1},j_{1}},...,\ast_{i_{r},j_{r}}))g)&=(\mathbf{D}_{\psi}f)(\rho(p(\ast_{i_{1},j_{1}},...,\ast_{i_{r},j_{r}}))g\\
&\hspace{5em}+f(\sigma(p(\ast_{i_{1},j_{1}},...,\ast_{i_{r},j_{r}})))\mathbf{D}_{\psi}g.
\end{align*}
Equally, take $p(\star_{i_{1},j_{1}},...,\star_{i_{r},j_{r}})\in\C[\{\star_{i_{a},j_{a}}\}_{a=1}^r]$. Then
\begin{align*}
\mathbf{D}_{\psi}(f(p(\star_{i_{1},j_{1}},...,\star_{i_{r},j_{r}}))g)&=(\mathbf{D}_{\psi}f)(\rho(p(\star_{i_{1},j_{1}},...,\star_{i_{r},j_{r}}))g\\
&\hspace{5em}+f(\sigma(p(\star_{i_{1},j_{1}},...,\star_{i_{r},j_{r}})))\mathbf{D}_{\psi}g.
\end{align*}
\end{theorem}

Finally, the following result is easily deduced.
\begin{theorem}
The universal algebras 
\begin{align*}
&\ward_{\psi,\C}[[x]]_{p(\ast_{i_{1},j_{1}},...,\ast_{i_{r},j_{r}})},\hspace{0.5em}\ward_{\psi,\C}[[x]]_{\ast_{\infty,\infty}},\\
&\ward_{\psi,\C}[[x]]_{p(\star_{i_{1},j_{1}},...,\star_{i_{r},j_{r}})},\hspace{0.5em}
\ward_{\psi,\C}[[x]]_{\star_{\infty,\infty}}.
\end{align*}
are differential.
\end{theorem}

\section{Leibniz's general $\psi$-rule}

In this section we will find the $n$-th $\psi$-derivative of the ordinary product of $\psi$-exponential generating functions in $\ward_{\psi,\C}[[x]]$. For the ordinary derivative it is known that Leibniz's general rule is
\begin{equation*}
    \mathbf{D}^{n}(f(x)g(x))=\sum_{k=0}^{n}\binom{n}{k}\mathbf{D}^{n-k}f(x)\mathbf{D}^{k}g(x)
\end{equation*}
and the $q$-analogue of this rule is
\begin{equation*}
    \mathbf{D}_{q}^{n}(f(x)g(x))=\sum_{k=0}^{n}\binom{n}{k}_{q}\mathbf{D}_{q}^{n-k}f(q^{k}x)\mathbf{D}_{q}^{k}g(x).
\end{equation*}
The idea of our construction of a general product $\psi$-rule is to use an analogous of binomial coefficients defined on $\C_{\psi}^{\ast}$ such that we can recover the above identities.

\begin{definition}\label{def_ope_binom}
We define the binomial operator $\superbinom{n}{k}_{\ast}\in\C_{\psi}^{\ast}$ as the operator $\superbinom{n}{k}_{\ast}:\ward_{\psi,\C}[[x]]\times\ward_{\psi,\C}[[x]]\rightarrow\ward_{\psi,\C}[[x]]$ satisfying
\begin{eqnarray}
    \superbinom{n}{0}_{\ast}&=&\ast_{\infty,\infty},\ n\geq0,\\
    \superbinom{n}{k}_{\ast}&=&\rho\superbinom{n-1}{k}_{\ast}\boxplus\sigma\superbinom{n-1}{k-1}_{\ast}\label{eqn_recu_sbinom},\ \ 1\leq k\leq n-1,\ n\geq2\\
    \superbinom{n}{n}_{\ast}&=&\prod_{i=1}^{n}\ast_{i,0},\ n\geq1.
\end{eqnarray}
\end{definition}

In Table 1 we can see some values of binomial operators.

\begin{table}[H]
    \centering
    \begin{tabular}{l|c|c|c|c|c}
    \backslashbox{$n$}{$k$}& 0&1 & 2& 3   \\\hline
         0& $\ast_{\infty,\infty}$&  &  & \\\hline
         1& $\ast_{\infty,\infty}$&  $\ast_{1,0}$&  &\\\hline
         2& $\ast_{\infty,\infty}$&  $\ast_{1,0}\boxplus\ast_{2,1}$&  $\ast_{1,0}\ast_{2,0}$&  \\\hline
         3& $\ast_{\infty,\infty}$&  $\ast_{1,0}\ast_{2,1}\ast_{3,2}$&  $\ast_{1,0}\ast_{2,0}\boxplus\ast_{1,0}\ast_{3,1}\boxplus\ast_{2,1}\ast_{3,1}$&  $\ast_{1,0}\ast_{2,0}\ast_{3,0}$\\\hline 
    \end{tabular}
    \caption{$\C_{\psi}^{\ast}$-Analogue of Pascal's Triangle}
    \label{tab:my_label}
\end{table}

In the following theorem we will give the $n$-th derivative of product $fg$. 

\begin{theorem}\label{theo_gen_leibniz}
Take $f,g\in\ward_{\psi,\C}[[x]]$. Then the $\psi$-rule general of Leibniz is
\begin{equation}
    \mathbf{D}_{\psi}^{n}(f(x)g(x))=\sum_{k=0}^{n}\superbinom{n}{k}_{\ast}(\mathbf{D}_{\psi}^{n-k}f(x),\mathbf{D}_{\psi}^{k}g(x))
\end{equation}
\end{theorem}
\begin{proof}
The proof will be by induction. Assume true for $n$ and let us compute the $\psi$-derivative of $\mathbf{D}_{\psi}^{n}(f(x)g(x))$. Applying the Theorem \ref{theo_der_prod_mu_} to the product 
\begin{equation*}
\superbinom{n}{k}_{\ast}(\mathbf{D}_{\psi}^{n-k}f(x),\mathbf{D}_{\psi}^{k}g(x))    
\end{equation*}
we obtain
\begin{align*}
    \mathbf{D}_{\psi}^{n+1}(f(x)g(x))&=\mathbf{D}_{\psi}\left(\sum_{k=0}^{n}\superbinom{n}{k}_{\ast}(\mathbf{D}_{\psi}^{n-k}f(x),\mathbf{D}_{\psi}^{k}g(x))\right)\\
    &=\mathbf{D}_{\psi}\left(\sum_{k=0}^{n}\mathbf{D}_{\psi}^{n-k}f(x)\superbinom{n}{k}_{\ast}\mathbf{D}_{\psi}^{k}g(x)\right)\\
    &=\sum_{k=0}^{n}\Bigg(\mathbf{D}_{\psi}^{n+1-k}f(x)\rho\superbinom{n}{k}_{\ast}\mathbf{D}_{\psi}^{k}g(x)
    \hspace{0em}+\mathbf{D}_{\psi}^{n-k}f(x)\sigma\superbinom{n}{k}_{\ast}\mathbf{D}_{\psi}^{k+1}g(x)\Bigg)\\
    &=\sum_{k=0}^{n}\mathbf{D}_{\psi}^{n+1-k}f(x)\rho\superbinom{n}{k}_{\ast}\mathbf{D}_{\psi}^{k}g(x)
    \hspace{0em}+\sum_{k=0}^{n}\mathbf{D}_{\psi}^{n-k}f(x)\sigma\superbinom{n}{k}_{\ast}\mathbf{D}_{\psi}^{k+1}g(x).
\end{align*}
Now we rewrite the second series for $k$ and draw summands from the two series, i.e,
\begin{align*}
&\mathbf{D}_{\psi}^{n+1}(f(x)g(x))\\
    &\hspace{3em}=\sum_{k=0}^{n}\mathbf{D}_{\psi}^{n+1-k}f(x)\rho\superbinom{n}{k}_{\ast}\mathbf{D}_{\psi}^{k}g(x)
    \hspace{0em}+\sum_{k=1}^{n+1}\mathbf{D}_{\psi}^{n+1-k}f(x)\sigma\superbinom{n}{k-1}_{\ast}\mathbf{D}_{\psi}^{k}g(x)\\
    &\hspace{3em}=\mathbf{D}_{\psi}^{n+1}f(x)\rho\superbinom{n}{0}_{\ast}g(x)
    \hspace{0em}+\sum_{k=1}^{n}\mathbf{D}_{\psi}^{n+1-k}f(x)\rho\superbinom{n}{k}_{\ast}\mathbf{D}_{\psi}^{k}g(x)\\
    &\hspace{9em}+\sum_{k=1}^{n}\mathbf{D}_{\psi}^{n+1-k}f(x)\sigma\superbinom{n}{k-1}_{\ast}\mathbf{D}_{\psi}^{k}g(x)
    \hspace{0em}+f(x)\sigma\superbinom{n}{n}_{\ast}\mathbf{D}_{\psi}^{n+1}g(x).
\end{align*}
As 
$$\rho\superbinom{n}{0}_{\ast}=\rho(\ast_{\infty,\infty})=\superbinom{n+1}{0}_{\ast}$$ 
and
$$\sigma\superbinom{n}{n}_{\ast}=\sigma\left(\prod_{i=1}^{n}\ast_{i,0}\right)=\prod_{i=1}^{n}\ast_{i+1,0}\ast_{1,0}=\superbinom{n+1}{n+1}_{\ast},$$
then
\begin{align*}
    &\mathbf{D}_{\psi}^{n+1}(f(x)g(x))\\
    &\hspace{3em}=\mathbf{D}_{\psi}^{n+1}f(x)\superbinom{n+1}{0}_{\ast}g(x)
    \hspace{0em}+\sum_{k=1}^{n}\mathbf{D}_{\psi}^{n+1-k}f(x)\rho\superbinom{n}{k}_{\ast}\mathbf{D}_{\psi}^{k}g(x)\\
    &\hspace{5em}+\sum_{k=1}^{n}\mathbf{D}_{\psi}^{n+1-k}f(x)\sigma\superbinom{n}{k-1}_{\ast}\mathbf{D}_{\psi}^{k}g(x)
    \hspace{0em}+f(x)\superbinom{n+1}{n+1}_{\ast}\mathbf{D}_{\psi}^{n+1}g(x)\\
    &\hspace{3em}=\superbinom{n+1}{0}_{\ast}(\mathbf{D}_{\psi}^{n+1}f(x),g(x))\\
    &\hspace{5em}+\sum_{k=1}^{n}\left(\rho\superbinom{n}{k}_{\ast}
    \boxplus\sigma\superbinom{n}{k-1}_{\ast}\right)(\mathbf{D}_{\psi}^{n+1-k}f(x),\mathbf{D}_{\psi}^{k}g(x))\\
    &\hspace{5em}+\superbinom{n+1}{n+1}_{\ast}(f(x),\mathbf{D}_{\psi}^{n+1}g(x)).
\end{align*}
Then by definition of $\superbinom{n}{k}_{\ast}$ we get to
\begin{align*}
&\mathbf{D}_{\psi}^{n+1}(f(x)g(x))\\
    &\hspace{3em}=\superbinom{n+1}{0}_{\ast}(\mathbf{D}_{\psi}^{n+1}f(x),g(x))
    \hspace{0em}+\sum_{k=1}^{n}\superbinom{n+1}{k}_{\ast}(\mathbf{D}_{\psi}^{n+1-k}f(x),\mathbf{D}_{\psi}^{k}g(x))\\
    &\hspace{9em}+\superbinom{n+1}{n+1}_{\ast}(f(x),\mathbf{D}_{\psi}^{n+1}g(x))\\
    &\hspace{3em}=\sum_{k=0}^{n+1}\superbinom{n+1}{k}_{\ast}(\mathbf{D}_{\psi}^{n+1-k}f(x),\mathbf{D}_{\psi}^{k}g(x))
\end{align*}
as it was intended to be shown.
\end{proof}

In the following propositions we will find special values of $\superbinom{n}{k}$. The proofs will be done by induction on $n$.

\begin{proposition}\label{prop_n_1}
For all $n\geq1$
\begin{equation}
    \superbinom{n}{1}_{\ast}=\bigboxplus_{i=0}^{n-1}\ast_{i+1,i}.
\end{equation}
\end{proposition}
\begin{proof}
Suppose that $\superbinom{n}{1}_{\ast}=\boxplus_{i=0}^{n-1}\ast_{i+1,i}$. Using the Definition \ref{def_ope_binom} we obtain.
\begin{eqnarray*}
    \superbinom{n+1}{1}_{\ast}(f,g)&=&
    \left(\sigma\superbinom{n}{0}_{\ast}\boxplus\rho\superbinom{n}{1}_{\ast}\right)(f,g)\\
    &=&\left(\sigma(\ast_{\infty,\infty})\boxplus\rho\left(\bigboxplus_{i=0}^{n-1}\ast_{i+1,i}\right)\right)(f,g)\\
    &=&\left(\ast_{\infty,\infty}\ast_{1,0}\boxplus\bigboxplus_{i=0}^{n-1}\ast_{i+2,i+1}\right)(f,g)\\
    &=&\left(\ast_{1,0}\boxplus\bigboxplus_{i=1}^{n}\ast_{i+1,i}\right)(f,g)\\
    &=&\left(\bigboxplus_{i=0}^{n}\ast_{i+1,i}\right)(f,g).
\end{eqnarray*}
Then $\superbinom{n+1}{1}_{\ast}=\bigboxplus_{i=0}^{n}\ast_{i+1,i}$.
\end{proof}

\begin{proposition}\label{prop_n_2}
For all $n\geq2$
\begin{equation}
    \superbinom{n}{2}_{\ast}=\bigboxplus_{i=0}^{n-2}\bigboxplus_{j=0}^{n-2-i}\ast_{i+j+2,i+j}\ast_{i+1,i}.
\end{equation}
\end{proposition}
\begin{proof}
Suppose true that 
$$\superbinom{n}{2}_{\ast}=\bigboxplus_{i=0}^{n-2}\bigboxplus_{j=0}^{n-2-i}\ast_{i+j+2,i+j}\ast_{i+1,i}.$$ Then
\begin{align*}
    \superbinom{n+1}{2}_{\ast}(f,g)
    &=\left(\sigma\superbinom{n}{1}_{\ast}\boxplus\rho\superbinom{n}{2}_{\ast}\right)(f,g)\\
    &=\left(\sigma\left(\bigboxplus_{j=0}^{n-1}\ast_{j+1,j}\right)\boxplus\rho\left(\bigboxplus_{i=0}^{n-2}\bigboxplus_{j=0}^{n-2-i}\ast_{i+j+2,i+j}\ast_{i+1,i}\right)\right)(f,g)\\
    &=\left(\bigboxplus_{j=0}^{n-1}\ast_{j+2,j}\ast_{1,0}\boxplus\bigboxplus_{i=0}^{n-2}\bigboxplus_{j=0}^{n-2-i}\ast_{i+j+3,i+j+1}\ast_{i+2,i+1}\right)(f,g)\\
    &=\left(\bigboxplus_{j=0}^{n-1}\ast_{j+2,j}\ast_{1,0}\boxplus\bigboxplus_{i=1}^{n-1}\bigboxplus_{j=0}^{n-1-i}\ast_{i+j+2,i+j}\ast_{i+1,i}\right)(f,g)\\
    &=\left(\bigboxplus_{i=0}^{n-1}\bigboxplus_{j=0}^{n-1-i}\ast_{i+j+2,i+j}\ast_{i+1,i}\right)(f,g).
\end{align*}
Thus $\superbinom{n+1}{2}_{\ast}=\bigboxplus_{i=0}^{n-1}\bigboxplus_{j=0}^{n-1-i}\ast_{i+j+2,i+j}\ast_{i+1,i}$ and the result is true for all $n$.
\end{proof}


When $\superbinom{n}{k}_{\ast}$ is a product in $\ward_{\C}[[x]]$, then
\begin{equation*}
    \superbinom{n}{k}_{\ast}(f,g)=\binom{n}{k}fg.
\end{equation*}
and thus we recover the general ordinary Leibniz formula. In the following result we show that it is also possible to recover the $q$-analogue of that rule.
\begin{theorem}
If $\psi_{n}=[n]$, then 
\begin{equation}
    \superbinom{n}{k}_{\ast}(f(x),g(x))=\binom{n}{k}_{q}f(q^{k}x)g(x)
\end{equation}
\end{theorem}
\begin{proof}
Using recursively the Eq.(\ref{eqn_recu_sbinom}) we have
\begin{equation*}\label{eqn_recu}
    \superbinom{n}{k}_{\ast}=\bigboxplus_{i=1}^{k}\sigma^{i-1}\rho\superbinom{n-i}{k-i+1}_{\ast}\boxplus\sigma^{k}\superbinom{n-1}{0}_{\ast}.
\end{equation*}
First we calculate $\sigma^{k}\superbinom{n-1}{0}_{\ast}$. We have
\begin{eqnarray*}
    \sigma^{k}\superbinom{n-1}{0}_{\ast}(f(x),g(x))&=&(\sigma^{k}\ast_{\infty,\infty})(f(x),g(x))\\
    &=&(\sigma^{k-1}\ast_{1,0})(f(x),g(x))\\
    &=&(\sigma^{k-2}\ast_{2,0}\ast_{1,0})(f(x),g(x))\\
    &=&\left(\prod_{i=1}^{k}\ast_{i,0}\right)(f(x),g(x))\\
    &=&f(q^{k}x)g(x).
\end{eqnarray*}
By a similar calculation using Propositions \ref{prop_n_1} and \ref{prop_n_2} it is proved that
\begin{eqnarray*}
\sigma^{k-1}\rho\superbinom{n-k}{1}_{\ast}(f(x),g(x))&=&q[n]f(q^{k}x)g(x),\\ \sigma^{k-2}\rho\superbinom{n-k+1}{2}_{\ast}(f(x),g(x))&=&q^{2}\binom{n-k+1}{2}_{q}f(q^{k}x)g(x),
\end{eqnarray*}
and in general
\begin{equation*}
    \rho\superbinom{n-1}{k}_{\ast}(f(x),g(x))=q^{k}\binom{n-1}{k}_{q}f(q^{k}x)g(x).
\end{equation*}
Now for Eq.(\ref{eqn_recu})
\begin{eqnarray*}
\superbinom{n}{k}_{\ast}(f(x),g(x))&=&\left(\sum_{i=1}^{k}q^{k-i+1}\binom{n-i}{k-i+1}_{q}+\binom{n-1}{0}_{q}\right)f(q^{k}x)g(x)\\
&=&\binom{n}{k}_{q}f(q^{k}x)g(x)
\end{eqnarray*}
where we have used recursively Eq.(\ref{eqn_psi_binom}) with $\psi_{n}=[n]$. 
\end{proof}
From the above theorem and Theorem \ref{theo_gen_leibniz} it follows that
\begin{align*}
    \mathbf{D}_{q}^{n}(f(x)g(x))&=\sum_{k=0}^{n}\superbinom{n}{k}_{\ast}(\mathbf{D}_{q}^{n-k}f(x),\mathbf{D}_{q}^{k}g(x))\\
    &=\sum_{k=0}^{n}\binom{n}{k}_{q}\mathbf{D}_{q}^{n-k}f(q^{k}x)\mathbf{D}_{q}^{k}g(x).
\end{align*}

We conclude by showing the following example.
\begin{example}
We want to calculate the $n$-th derivative of the function $xe_{\psi}^{x}$, where
\begin{equation*}
    e_{\psi}^{x}=\sum_{n=0}^{\infty}\frac{x^n}{\psi_{n}!}.
\end{equation*}
Using the Theorem \ref{theo_gen_leibniz} we have
\begin{eqnarray*}
    \mathbf{D}_{\psi}^{n}(xe_{\psi}^{x})&=&\sum_{k=0}^{n}\superbinom{n}{k}_{\ast}(\mathbf{D}_{\psi}^{n-k}x,\mathbf{D}_{\psi}^{k}e_{\psi}^{x})\\
    &=&\superbinom{n}{n}_{\ast}(x,\mathbf{D}_{\psi}^{n}e_{\psi}^{x})+\superbinom{n}{n-1}_{\ast}(\mathbf{D}_{\psi}x,\mathbf{D}_{\psi}^{n-1}e_{\psi}^{x})\\
    &=&x\left(\prod_{i=1}^{n}\ast_{i,0}\right)e_{\psi}^{x}+1\superbinom{n}{n-1}_{\ast}e_{\psi}^{x}.
\end{eqnarray*}
When $\psi_{m}=m$, the $n$-th derivative of $xe^{x}$ is $\mathbf{D}^{n}(xe^{x})=(n+x)e^{x}$. When $\psi_{m}=[m]$, then. 
$\mathbf{D}_{q}^{n}(xe_{q}^{x})=(q^{n}x+[n])e_{q}^{x}$.
\end{example}

\section{Conclusions and Perspectives}
As we could notice in this paper, it was necessary to construct the Ward-Fonten\'e universal differential algebra in order to find an $\psi$-analog of the product rule and of the general Leibniz rule. These $\psi$-analogues are compatible with the existing product rules in Newton's calculus, the $q$-calculus and with the $(p,q)$-calculus and allow us to obtain the respective rules in the simplicial polytopic calculus and in the calculus on Bell numbers and in general any other calculus on sequences. The main purpose for introducing the simplicial polytopic calculus is because of its relation to special functions. For example, the equation $\mathbf{D}_{T}y=y$ has solution
\begin{equation*}
    y(x)=c_{1}\frac{I_{1}(2\sqrt{2}\sqrt{x})}{\sqrt{2}\sqrt{x}}+c_{2}\frac{\sqrt{2}K_{1}(2\sqrt{2}\sqrt{x})}{\sqrt{x}},
\end{equation*}
where $I_{n}(x)$ and $K_{n}(x)$ are the modified Bessel functions of the first and second kind, respectively. Then the triangular numbers are related to the modified Bessel functions. Moreover the $\mathrm{T}$-analog of the exponential function is the function $y(x)$ above and we may then be interested in studying the $T$-analogs of a huge list of applications of the ordinary exponential function and the relation of these analogs to the modified Bessel functions. Similarly, we find that the solutions of the differential equations $\mathbf{D}_{S_{r}}y=y$, $r\geq3$, are related to the Meijer $G$-functions. On the other hand, it will be very interesting to find the solution of the equation $\mathbf{D}_{B}y=y$ in terms of known functions, as well as possible applications of this solution. Since Bell numbers have their main applications in Combinatorics, we think that a calculus on these numbers has mainly applications in this area of mathematics. Finally, we want to study properties and applications of the binomial operator in $\C_{\psi}^{\ast}$ that are analogous to those of the ordinary binomial coefficients.


\end{document}